\documentclass{amsart}
\usepackage{amsmath, amssymb, amsthm, mathtools,xcolor}
\usepackage{graphicx}
\usepackage[colorlinks=true, linkcolor=red, citecolor=red, urlcolor=blue]{hyperref}
\usepackage{comment}
\usepackage{setspace}
\usepackage[top=1.5in, bottom=1.5in, left=1.3in, right=1.3in]{geometry}

\newtheorem{theorem}{Theorem}[section]
\newtheorem{lemma}[theorem]{Lemma}
\newtheorem{corollary}[theorem]{Corollary}
\newtheorem{proposition}[theorem]{Proposition}
\newtheorem{conjecture}[theorem]{Conjecture}
\newtheorem{definition}[theorem]{Definition}

\newtheorem{remark}[theorem]{Remark}

\theoremstyle{definition}


\numberwithin{equation}{section}

\title{Simultaneous Periods for Families of Rational Maps Modulo Primes}
\author{Bhawesh Mishra}
\address{Department of Mathematical Sciences, University of Memphis, Memphis, TN, 38152}
\email{bhaweshmishra2024@gmail.com}

\begin{document}

\begin{abstract}
Let \(K\) be a number field, and \(\varphi_{1},\ldots,\varphi_{g}\in K(t)\) be finitely many rational maps, each of degree at least \(2\). We first show that for generic finite sets \(\mathcal{A}_{1},\ldots,\mathcal{A}_{g}\) consisting entirely of points that are not \(\varphi_{i}\)-periodic, there exists a set of primes \(\mathfrak p\) of \(K\) of positive density such that for each \(\mathcal{A}_{i}\) and every \(\alpha\in\mathcal{A}_i\), \(\alpha\) is not \(\varphi_i\)-periodic modulo \(\mathfrak p\). The notion of genericity used here is defined in terms of the associated arboreal fields and is sharper than those previously used in the literature. Leveraging our proof in the generic case, we then show that the same conclusion holds for most \textit{expected} cases of non-generic sets \(\mathcal{A}_{i}\). Finally, we apply our result to confirm the dynamical Mordell--Lang conjecture for coordinate-wise actions of a class of maps that includes rational maps that are generic in this sense.
\end{abstract}

\maketitle

\section{Introduction}\label{sec:intro}
Throughout this article, \(K\) will denote a number field. Given a rational map \(\varphi:\mathbb{P}^{1}_{K}\longrightarrow\mathbb{P}^{1}_{K}\) of degree at least \(2\) and integer \(m\geq 1\), we will write \(\varphi^{m}\) for the \(m\)-th iterate of \(\varphi\) under composition. The \textit{forward orbit} of a point \(\alpha\in\mathbb{P}^{1}(K)\) under \(\varphi\) is the set \( \mathcal{O}^{+}(\alpha) := \{\varphi^{m}(\alpha):m\geq 0\}\). Similarly, the \textit{backward orbit} of \(\alpha\) is the set \(\mathcal{O}^{-}(\alpha) := \{\beta\in\mathbb{P}^{1}(\overline{K})\,:\,\varphi^{m}(\beta)=\alpha\text{ for some }m\geq 0\}\). We will say that \(\alpha\) is \textit{\(\varphi\)-periodic} if \(\varphi^{m}(\alpha)=\alpha\) for some \(m\geq 1\). More generally, we will say that \(\alpha\) is \textit{\(\varphi\)-preperiodic} if its forward orbit is finite and \(\alpha\) is \textit{exceptional} if its backward orbit is finite.

\color{black}
\subsection{Motivation}
Given two points \(\alpha,\beta\in\mathbb{P}^{1}(K)\) such that \(\beta\) is not \(\varphi\)-preperiodic and \(\alpha\) is not in the forward orbit of \(\beta\) under \(\varphi\), one might ask how many primes \(\mathfrak{p}\) of \(K\) there are such that \(\alpha\) \textit{is} in the forward orbit of \(\beta\) under \(\varphi\) modulo \(\mathfrak{p}\). The fact that there are infinitely many such \(\mathfrak{p}\) unless \(\alpha\) is exceptional for \(\varphi\) follows from \cite[Lemma 4.1]{etaleDML}. Then, an application of \cite[Theorem 2.2]{Sil93}) also shows that there are infinitely many \(\mathfrak{p}\) such that \(\alpha\) is \textit{not} in the forward orbit of \(\beta\) modulo \(\mathfrak{p}\). It is known that the set \(\mathcal{S}\) of primes \(\mathfrak{p}\) such that \(\alpha\) is in the forward orbit of \(\beta\) modulo \(\mathfrak{p}\) has density zero in some cases (see \cite{OdoniIterates} and \cite{Jon08} amongst others). However, there are cases when \(\mathcal{S}\) has positive density. More generally, one may expect such examples for \textit{exceptional maps} (see \cite{GTZ07}). In this context, the set of primes such that \(\alpha\) is not in the forward orbit of \(\beta\) always has positive density - the most general case of this result as established in \cite{BGHKST} which we state in Theorem \ref{thm:BGHKST}. 

Isolating the primes \(\mathfrak{p}\) such that \(\alpha\) is in the forward orbit of \(\beta\) has implications for what is known as the dynamical Mordell-Lang conjecture, which is a dynamic analogue to the classical cyclic-case of the Mordell-Lang conjecture established in \cite{Faltings, McQuillan, Vojta} \cite{GT09, GTZ08}. 
\begin{conjecture}{\textbf{(Dynamical Mordell-Lang Conjecture})}
Let \(\Phi\) be an endomorphism of a quasi-projective variety  \(X\) defined over \(\mathbb{C}\), \(V \subset X\) be a closed subvariety and \(x \in X(\mathbb{C})\). Then, the set of positive integers \(n\) such that \(\Phi^{n}(x) \in V(\mathbb{C})\) is a union of finitely many arithmetic progressions\(\{an+b\}_{n\in\mathbb{N}}\) with $a, b\geq 0$.
\label{con:DMLC}
\end{conjecture}

The conjecture \ref{con:DMLC} for \'etale maps was first established by Bell, Ghioca and Tucker in \cite{etaleDML}, and it was established for endomorphisms of \(\mathbb{A}^{2}\) by Xie in \cite{Xie17}. In this article however, we will focus on the case of Conjecture \ref{con:DMLC} when \(X = (\mathbb{P}^{1})^{g}\) and \(\Phi := (\varphi_{1}, \varphi_{2}, \ldots, \varphi_{g})\) will be the coordinate-wise action for rational maps \(\varphi_{i}\), which is known as the \textit{split} case. The following result, first established in \cite{BGKT12}, gives sufficient conditions for the conclusions of Conjecture \ref{con:DMLC} in terms of non-intersection of orbits with residue classes of attracting periodic point. 

\begin{proposition}[{\cite[Theorem~3.4]{BGKT12}}]
Let $V$ be a subvariety of $(\mathbb{P}^{1})^{g}$ defined over $\mathbb{C}_{p}$, let
$f_{1},\dots,f_{g} \in \mathbb{C}_{p}(t)$ be rational functions of good reduction on $\mathbb{P}^{1}$,
and let $\Phi$ denote the coordinatewise action of $(f_{1},\dots,f_{g})$ on
$(\mathbb{P}^{1})^{g}$. Let $\mathcal{O}$ be the $\Phi$-orbit of a point
\[
\alpha = (x_{1},\dots,x_{g}) \in (\mathbb{P}^{1}(\mathbb{C}_{p}))^{g},
\]
and suppose that for each $i$, the orbit $\mathcal{O}_{f_{i}}(x_{i})$ does not intersect
the residue class of any attracting $f_{i}$-periodic point. Then
$V(\mathbb{C}_{p}) \cap \mathcal{O}$ is a union of at most finitely many orbits of the form
\[
\{\Phi^{nk+\ell}(\alpha)\}_{n \ge 0}
\]
for nonnegative integers $k$ and $\ell$. \label{prop:orbitAP}
\end{proposition}
Readers may consult \cite{DMLBook} for a broad and accessible discussion on the use of \(p\)-adic analytic approach towards establishing the dynamical Mordell-Lang conjecture, and \cite{Amerik11} for analogous result for finite extensions of \(\mathbb{Q}_{p}\). Using the main result in our article, we establish the following case of dynamical Mordell-Lang conjecture for the coordinatewise action of linearly-disjoint system of rational maps (please see Definition \ref{def:linearly-disjoint-system} for definition of linearly-disjoint system).

\begin{theorem}\label{thm:DML}
Let $K$ be a number field, \(g \in\mathbb{N}\), \(V \subset (\mathbb{P}^{1})^{g}\) be a subvariety defined over \(K\), let \(x = (x_{1}, \ldots, x_{g})\in(\mathbb{P}^{1})^{g}(K)\). Moreover, let \(\Phi := (\varphi_{1}, \ldots, \varphi_{g})\) act on \((\mathbb{P}^{1})^{g}\) coordinatewise, where each \(\varphi_{i}\in K(t)\) is a rational function of degree \(\geq 2\). Furthermore, let \(\mathcal{S}_{i}\) be the set of critical points of \(\varphi_{i}\) and assume that the set \(\big\{(\varphi_{i}, \mathcal{S}_{i})\big\}_{i=1}^{g}\) forms an almost linearly-disjoint system.  Then, the set of integers $n \in \mathbb{N}$ such that $\Phi^{n}(x) \in V(\overline{K})$ is a union of finitely many arithmetic progressions $\{nk + \ell\}_{n \in \mathbb{N}}$, where $k,\ell \ge 0$ are nonnegative integers.  
\end{theorem}

We note that a result analogous to Theorem \ref{thm:DML} for \textit{generic} polynomial maps \(f(x)\in K[x]\) of degree \(d \geq 2\) was established in \cite{Juul}, where the notion of genericity is that the critical points, and constant term, of polynomials lie in a Zariski-dense subset of the appropriate parameter space. This assumption in \cite{Juul} allows one to conclude that the Galois group of the splitting field of $(f^{\prime}\circ f)$ is isomorphic to the direct product of \((d-1)\) copies of $n$-fold iterated wreath product of \(S_{d}\). The notion of genericity used in this article (called \textit{almost linear-disjointness} and defined in \ref{def:linearly-disjoint-system}) does not involve any such assumption on Galois groups. Similarly, a more general case of the dynamical Mordell-Lang conjecture for generic endomorphisms of \(\mathbb{P}_{M}^{n}\), where \(M\) is an algebraically closed field of characteristic zero, was established by Fakhruddin in \cite{Fak14}. In \cite{Fak14}, the notion of genericity is based on avoidance of countably many proper subvarieties in the moduli space of morphisms of \(\mathbb{P}_{M}^{n}\) to itself. As promised earlier in the introduction, we will deal with a more general problem and then establish Theorem \ref{thm:DML} as a consequence. In this context, consider the following result that was first established in \cite{BGHKST}.

\begin{theorem}\label{thm:BGHKST}
Let \(K\) be a number field, and let \(\varphi_{1},\ldots,\varphi_{g}:\mathbb{P}_{K}^{1}\longrightarrow\mathbb{P}_{K}^ {1}\) be rational maps of degree at least 2. Let \({\mathcal{A}}_{1},\ldots,{\mathcal{A}}_{g}\) be finite subsets of \(\mathbb{P}^{1}(K)\) such that at most one set \({\mathcal{A}}_{i}\) contains a point that is not \(\varphi_{i}\)-preperiodic, and there is at most one such point in that set \({\mathcal{A}}_{i}\). Let \({\mathcal{T}}_{1},\ldots,{\mathcal{T}}_{g}\) be finite subsets of \(\mathbb{P}^{1}(K)\) such that no \({\mathcal{T}}_{i}\) contains any \(\varphi_{i}\)-preperiodic points. Then there is a positive integer \(M\) and a set of primes \({\mathcal{P}}\) of \(K\) having positive density such that for any \(i=1,\ldots,g\), any \(\gamma\in{\mathcal{T}}_{i}\), any \(\alpha\in{\mathcal{A}}_{i}\), any \(\mathfrak{p}\in{\mathcal{P}}\), and any \(m\geq M\),
\[
\varphi_{i}^{m}(\gamma)\not\equiv\alpha\pmod{\mathfrak{p}}.
\]
\end{theorem}
The following local-global principle regarding periodic points of rational maps is a consequence of Theorem \ref{thm:BGHKST}. 

\begin{proposition}\label{cor:nonper-red}
Let \(K\) be a number field, let \(\varphi:\mathbb{P}^{1}_{K}\longrightarrow\mathbb{P}^{1}_{K}\) be a rational function of degree at least 2, and let \(\alpha\in\mathbb{P}^{1}(K)\) be a non-periodic point of \(\varphi\). Then there is a positive density set of primes \(\mathfrak{p}\) of \(K\) at which \(\varphi\) has good reduction \(\varphi_{\mathfrak{p}}\) and such that the reduction of \(\alpha\) modulo \(\mathfrak{p}\) is not \(\varphi_{\mathfrak{p}}\)-periodic.
\end{proposition}

In \cite[Section 6, pp.658]{BGHKST}, the authors asked whether the restriction in Theorem \ref{thm:BGHKST}, that at most one set $\mathcal{A}_{i}$ contains at most one point that is not \(\varphi_{i}\)-preperiodic, can be relaxed (see \cite[Section 6, Question 6.1]{BGHKST}). Even for $g = 1$ and \(\mathcal{A}_{1} = \{\alpha_{1}, \alpha_{2}\}\) the methods in \cite{BGHKST} do not yield the desired conclusion, unless \(\mathcal{A}_{1} = \{\alpha_{1}, \alpha_{2}\}\) satisfies one of the following special arrangements:
\begin{enumerate}
    \item Exactly one of the element of \(\mathcal{A}_{1}\) is pre-periodic as in Theorem \ref{thm:BGHKST}.

    \item One element of \(\mathcal{A}_{1}\) lies in the forward orbit of the other. 
\end{enumerate}
This limitation stems from the fact that the Theorem \ref{thm:BGHKST} in \cite{BGHKST} employs \cite[Lemma 4.3]{BGKT12} which itself relies on Roth's theorem and results in \cite{Sil93}. It is not obvious how to extend the methods in either of those articles when \(\mathcal{A}_{1}\) does not satisfy any of the conditions above. One case in which the answer is known in slightly more generality is where \(\varphi(x) = x^{2} + c\) for \(c\not\in\{0, -1\}\) and \(\alpha_{1} = 0\) (see \cite[Section 6, pp. 658]{BGHKST}). In this article, we establish a multi-element version of Hasse-principle for periodic points, under the assumption that the fields generated by the iterated pre-images are eventually linearly disjoint over some sub-field. Please see \ref{def:linearly-disjoint-system} and \ref{def:good-red} for pertinent definitions of almost \(\varphi\)-disjoint and of good reduction.

\begin{theorem}\label{thm:multiperiodic}
Let $K$ be a number field, \(\varphi: \mathbb{P}^{1}_{K} \longrightarrow \mathbb{P}^{1}_{K}\) be a rational function of degree at least $2$. Let $\alpha_{1}, \alpha_{2}, \ldots, \alpha_{k}$ be finitely many non-periodic points of \(\varphi\) such that \(\{\alpha_{i}\}_{i=1}^{k}\) is almost \(\varphi\)-linearly disjoint. Then, there exists a positive density subset of primes $\mathfrak{p}$ of $K$ at which \(\varphi\) has good reduction \(\varphi_{\mathfrak{p}}\) at \(\mathfrak{p}\) such that \(\alpha_{i}\) modulo \(\mathfrak{p}\) is not \(\varphi_{\mathfrak{p}}\)-periodic for any \(1 \leq i \leq k\). 
\end{theorem}

In fact, Theorem \ref{thm:multiperiodic} holds in much greater generality for multiple rational maps as follows. 

\begin{theorem}\label{thm:multimultiperiodic}
Let $K$ be a number field, \(g \in\mathbb{N}\) and \(\varphi_{1}, \ldots, \varphi_{g}: \mathbb{P}^{1}_{K} \longrightarrow \mathbb{P}^{1}_{K}\) be finitely many rational functions, each of degree at least $2$. For every \(1 \leq i \leq g\), let \(\mathcal{S}_{i} = \big\{\alpha_{ij} \big\}_{j=1}^{n_{i}}\subset\mathbb{P}^{1}(K)\) be a set consisting entirely of points that are not \(\varphi_{i}\)-periodic. Furthermore, suppose that the set \(\big\{(\varphi_{i}, \mathcal{S}_{i})\big\}_{i=1}^{g}\) forms an almost linearly-disjoint system.  Then, there exists a positive density subset of primes $\mathfrak{p}$ of $K$ at which each of the \(\varphi_{i}\) has good reduction such that for every \(1 \leq i \leq g\) and for every \(1 \leq j \leq n_{i}\), \(\alpha_{ij}\) is not \(\varphi_{i}\)-periodic modulo \(\mathfrak{p}\).
\end{theorem}

Theorem \ref{thm:multimultiperiodic} leads to the following multi-map and multi-element generalization of Theorem \ref{thm:BGHKST}. 

\begin{theorem}\label{thm:main-intro}
Let \(K\) be a number field, let \(\varphi_{1},\ldots,\varphi_{g}:\mathbb{P}_{K}^{1}\longrightarrow\mathbb{P}_{K}^ {1}\) be rational maps of degree at least 2, and let \({\mathcal{A}}_{1},\ldots,{\mathcal{A}}_{g}\) be finite subsets of \(\mathbb{P}^{1}(K)\). Let \(\mathcal{B}_{i} \subseteq \mathcal{A}_{i}\) be the subset consisting of all the points in \(\mathcal{A}_{i}\) that are not \(\varphi_{i}\)-preperiodic, if there are any. Furthermore, suppose that the set \(\big\{(\varphi_{i}, \mathcal{B}_{i})\big\}_{i=1}^{g}\) forms an almost linearly-disjoint system. Let \({\mathcal{T}}_{1},\ldots,{\mathcal{T}}_{g}\) be finite subsets of \(\mathbb{P}^{1}(K)\) such that no \({\mathcal{T}}_{i}\) contains any \(\varphi_{i}\)-preperiodic points. Then, there is a positive integer \(M\) and a set of primes \({\mathcal{P}}\) of \(K\) with positive density such that for any \(i=1,\ldots,g\), any \(\gamma\in{\mathcal{T}}_{i}\), any \(\alpha\in{\mathcal{A}}_{i}\), any \(\mathfrak{p}\in{\mathcal{P}}\), and any \(m\geq M\),
\[
\varphi_{i}^{m}(\gamma)\not\equiv\alpha\pmod{\mathfrak{p}}.
\]
\end{theorem}

\color{black}

Firstly, we note that both Theorems \ref{thm:multimultiperiodic} and \ref{thm:main-intro} hold in greater generality as a consequence of our methods. We will state and establish them in Section \ref{sec:augmentation}. We also want to note that the proof of Theorem \ref{thm:main-intro} in our article follows a logically different order from the proof of Theorems \ref{thm:BGHKST} and \ref{thm:multiperiodic} in \cite{BGHKST}. For us, Theorem \ref{thm:main-intro} is a consequence of Theorem \ref{thm:multimultiperiodic}. This is because \cite{BGHKST} uses a result from \cite{BGKT12} which guarantees the existence of infinitely many primes modulo which a non-periodic element stays non-periodic. However, no such result exists for multiple maps and multiple elements. Therefore, in this article we have to first establish Theorem \ref{thm:multimultiperiodic} to obtain infinitely many primes modulo which finitely many elements stay non-perodic. Only then, we can establish Theorem \ref{thm:main-intro} and its application to dynamical Mordell-Lang conjecture. 

Another corollary of Theorem \ref{thm:main-intro} pertains to elliptic curves and we need to introduce Latt\'es map to state it. Suppose that \(E\) is an elliptic curve defined over a number field \(K\) and \(q\) is a prime. When \(E\) is written in Weierstrass form, we have a morphism \(x: E \longrightarrow \mathbb{P}^{1}_{K}\) that takes a point in \(E\) to its first coordinate. Let \([q] : E \longrightarrow E\) denote the multiplication by \(q\). Then, there exists a unique \(\varphi\in K(x)\) such that \(x \circ [q] = \varphi \circ x\). The map \(\varphi\) is called the Latt\'es map with respect to \(q\). Readers can consult \cite{Sil09} for further facts about Latt\'es map and elliptic curves in general. 

\begin{corollary}\label{cor:multiell-order}
Let \(K\) be a number field, \(g \in\mathbb{N}\), \(E_{1}, E_{2}, \ldots, E_{g}\) be finitely many elliptic curves defined over \(K\), \(q\) be a rational prime and \(\varphi_{1}, \varphi_{2}, \ldots, \varphi_{g}\) be the associated Latt\'es maps with respect to \(q\) respectively. For every \(1 \leq i \leq g\), let \(\mathcal{S}_{i} = \{Q_{ij}\}_{j=1}^{n_{i}}\subset E_{i}(K)\) be a subset containing finitely many non-torsion points. Furthermore, assume that \(\big\{\big(\varphi_{i}, \{x(Q_{ij})\}_{j=1}^{n_{i}}\big)\big\}_{i=1}^{g}\) forms an almost linearly-disjoint system. Then, for every positive integer \(n\) there exists a positive density set of primes \(\mathfrak{p}\) of \(K\) at which each of the \(E_{i}\) has good reduction such that for every \(1 \leq i \leq g\) and every \(1 \leq j \leq n_{i}\), the order of \((Q_{ij})_{\mathfrak{p}}\) in the finite group \((E_{i})_{\mathfrak{p}}(k_{\mathfrak{p}})\) is divisible by \(q^{n}\).
\end{corollary}

Corollary \ref{cor:multiell-order} is a weaker result than what is established by Pink for arbitrary abelian varieties (see \cite[Theorem 5.3]{Pin04}) and by Perucca for arbitrary semi-abelian varieties (see \cite[Proposition 4.2]{Perucca}). 

This article is organized into several short sections. The next section introduces notations and definitions used in this article. Section \ref{sec:prelim-result} contains preliminary facts and results that will be employed to establish most results in this article. Section \ref{sec:Proofs-Main-Results} is dedicated to establishing Theorems \ref{thm:multiperiodic}, \ref{thm:multimultiperiodic} and Corollary \ref{cor:multiell-order} respectively. The next section is fully dedicated to the main result of this article, Theorem \ref{thm:main-intro}. Section \ref{sec:DML} is dedicated to establishing Theorem \ref{thm:DML}. Section \ref{sec:augmentation} is dedicated to explaining why the results in this article invariant under augmentation of points with elements of forward orbits. This is the section where an unconditional case of dynamical Mordell-Lang conjecture for certain diagonal action of polynomials is established. Section \ref{sec:Arboreal} gives a pleasing combinatorial interpretation of main results in terms of arboreal Galois theory and the last section contains some final remarks. 

\section{Notation and terminology}\label{sec:prelim}
Given a number field \(K\) with algebraic closure \(\overline{K}\), \(\mathfrak{o}_{K}\) will denote the corresponding ring of algebraic integers. We will call a nonzero prime \(\mathfrak{p}\) of \(\mathfrak{o}_{K}\) simply a \textit{prime of \(K\)}, and we denote the corresponding residue field \(k_{\mathfrak{p}}:=\mathfrak{o}_{K}/\mathfrak{p}\). Upon fixing an isomorphism \(\pi\) from \({\mathbb{P}}^{1}_{K}\) to the generic fibre of \({\mathbb{P}}^{1}_{\mathfrak{o}_{K}}\), for each prime \(\mathfrak{p}\) of \(K\), and for each \(x\in{\mathbb{P}}^{1}(K)\), we denote by \(r_{\mathfrak{p}}(x)\) the intersection of the Zariski closure of \(\pi(x)\) with the fibre above \(\mathfrak{p}\) of \({\mathbb{P}}^{1}_{\mathfrak{o}_{K}}\). The resulting map \(r_{\mathfrak{p}}:{\mathbb{P}}^{1}(K)\longrightarrow{\mathbb{P}}^{1}(k_{ \mathfrak{p}})\) is the \textit{reduction map} at \(\mathfrak{p}\). We say that \(\alpha\in{\mathbb{P}}^{1}(K)\) is \textit{congruent} to \(\beta\in{\mathbb{P}^{1}}(K)\) modulo \(\mathfrak{p}\), and we write \(\alpha\equiv\beta\pmod{\mathfrak{p}}\), if \(r_{\mathfrak{p}}(\alpha)=r_{\mathfrak{p}}(\beta)\).

If \(\varphi:{\mathbb{P}}^{1}\to{\mathbb{P}}^{1}\) is a morphism defined over the field \(K\), then after fixing a choice of homogeneous coordinates, there are relatively prime homogeneous polynomials \(F,G\in K[X,Y]\) of the same degree \(d=\deg\varphi\) such that \(\varphi([X,Y])=[F(X,Y):G(X,Y)]\); note that \(F\) and \(G\) are uniquely defined up to a nonzero constant multiple. One defines the following notion of good reduction of \(\varphi\), first introduced by Morton and Silverman in \cite{MS94}.

\begin{definition}\label{def:good-red}
Let \(K\) be a number field, let \(\mathfrak{p}\) be a prime of \(K\), and let \(\mathfrak{o}_{\mathfrak{p}}\subseteq K\) be the corresponding local ring of integers. Let \(\varphi:{\mathbb{P}}^{1}\longrightarrow{\mathbb{P}}^{1}\) be a morphism over \(K\), given by \(\varphi([X,Y])=[F(X,Y):G(X,Y)]\), where \(F,G\in\mathfrak{o}_{\mathfrak{p}}[X,Y]\) are relatively prime homogeneous polynomials of the same degree such that at least one coefficient of \(F\) or \(G\) is a unit in \(\mathfrak{o}_{\mathfrak{p}}\). Let \(\varphi_{\mathfrak{p}}:=[F_{\mathfrak{p}},G_{\mathfrak{p}}]\), where \(F_{\mathfrak{p}},G_{\mathfrak{p}}\in k_{\mathfrak{p}}[X,Y]\) are the reductions of \(F\) and \(G\) modulo \(\mathfrak{p}\). We say that \(\varphi\) has good reduction at \(\mathfrak{p}\) if \(\varphi_{\mathfrak{p}}:{\mathbb{P}^{1}}(k_{\mathfrak{p}})\longrightarrow{\mathbb{P}^{1}}(k_{ \mathfrak{p}})\) is a morphism of the same degree as \(\varphi\).
\end{definition}

The map \(\varphi_{\mathfrak{p}}:{\mathbb{P}^{1}}(k_{\mathfrak{p}})\to{\mathbb{P}^{1}}(k_{ \mathfrak{p}})\) in Definition \ref{def:good-red} is an appropriate analogue of the reduction of \(\varphi\) modulo \(\mathfrak{p}\). If \(\varphi\in K[t]\) is a polynomial, there is an elementary criterion for good reduction: \(\varphi\) has good reduction at \(\mathfrak{p}\) if and only if all coefficients of \(\varphi\) are \(\mathfrak{p}\)-adic integers, and its leading coefficient is a \(\mathfrak{p}\)-adic unit. 

\subsection{Almost Linearly Disjoint System}
Given a rational map \(\varphi : \mathbb{P}^{1}_{K} \longrightarrow \mathbb{P}^{1}_{K}\) of degree \(d \geq 2\), \(\alpha\in\mathbb{P}^{1}(K)\) and \(n \geq 1\), define \(K_{n}(\varphi, \alpha) := K\big(\varphi^{-n}(\alpha)\big) \) to be the field obtained upon adjoining all the roots of \(\varphi^{n}(x) - \alpha\) to \(K\). Then, we define \(K_{\infty}(\varphi, \alpha) := \bigcup_{n=1}^{\infty} K_{n}(\varphi, \alpha)\) which will be called the \textit{arboreal field} associated with \(\alpha\) over \(K\).

\begin{definition}\label{def:linearly-disjoint-system}
Let \(K\) be a field, let \(\varphi_{1}, \ldots, \varphi_{g} : \mathbb{P}^{1}_{K} \longrightarrow \mathbb{P}^{1}_{K} \) be finitely many rational maps of degree at least $2$ and let $\mathcal{S}_{1}, \ldots, \mathcal{S}_{g}$ be finite subsets of \(\mathbb{P}^{1}(K)\). We say that \(\big\{(\varphi_{i}, \mathcal{S}_{i})\big\}_{i=1}^{g}\) forms an \textit{almost linearly-disjoint system} if for every extension \(E\) of \(K\), there exists arbitrarily large positive integer \(m\) and corresponding finite extensions \(L := L^{(m)}\) of \(E\) satisfying the following:
    \begin{enumerate}
        \item For every \(1 \leq i \leq g\) and every \(\alpha\in\mathcal{S}_{i}\), \(L \subsetneq L_{m}\big(\varphi_{i}, \alpha\big)\).
in

        \item For every \(1 \leq i \leq g\), the fields \(\{L_{m}(\varphi_{i}, \alpha)\}_{\alpha\in\mathcal{S}_{i}}\) are linearly disjoint over \(L\).

        \item If \(g \geq 2\), the compositum fields \(\Big\{\prod_{\alpha\in\mathcal{S}_{i}} L_{m}(\varphi_{i}, \alpha)\Big\}_{i=1}^{g}\) are linearly disjoint over \(L\).
    \end{enumerate} 
In the case of \(g = 1\), when \(\{(\varphi_{1}, \mathcal{S}_{1})\}\) forms a linearly-disjoint system, we will say that \(\mathcal{S}_{1}\) is almost \(\varphi_{1}\)-linearly disjoint. 
\end{definition}

\begin{remark}\label{rem:linearly-disjoint-1}
It is a straightforward consequence of the definition \ref{def:linearly-disjoint-system} that if \(\{(\varphi_{i}, \mathcal{S}_{i})\}_{i=1}^{g}\) forms a linearly-disjoint system when \(\mathcal{S}_{i}\) are considered as a subset of \(\mathbb{P}^{1}(K)\), then \(\{(\varphi_{i}, \mathcal{S}_{i})\}_{i=1}^{g}\) forms a linearly-disjoint system when \(\mathcal{S}_{i}\) are considered as a subset of \(\mathbb{P}^{1}(L)\), for any finite extension \(L\) of \(K\). 
\end{remark}

\color{black}
\subsection{Some Examples of Linearly Disjoint Systems} Clearly any system that is linearly disjoint over \(K\) automatically forms an almost linearly-disjoint system. We will discuss some examples of almost linearly disjoint system that are not linearly-disjoint over \(K\). 

Take \(K = \mathbb{Q}\), \(\varphi(x) = x^{2}\) and \(\mathcal{S} = \{ 3, 5, 7\}\). Then, for any \(\alpha\in \mathcal{S}\) and any \(m\geq 2\), we have that \(K_{m}(\varphi, \alpha) = \mathbb{Q}\big(\zeta_{2^{m}}, \alpha^{\frac{1}{2^{m}}}\big),\) where \(\zeta_{2^{m}}\) is the primitive \(2^{m}\)-th root of unity. Then, the fields \(\big\{K_{m}(\varphi, \alpha)\big\}_{\alpha\in\mathcal{S}}\) are not linearly-disjoint over \(\mathbb{Q}\), since their intersection contains \(\mathbb{Q}(\zeta_{2^{m}})\). However, they are almost linearly-disjoint because the fields \(\big\{K_{m}(\varphi, \alpha)\big\}_{\alpha\in\mathcal{S}}\) are linearly-disjoint over a proper subfield \(\mathbb{Q}(\zeta_{2^{m}})\), and hence almost linearly-disjoint. A more concrete and non-trivial example regarding almost linear-disjoint system is implied by a recent result of Han and Tucker (see \cite[Theorem 5.4]{HanTucker}).

\begin{theorem}
Let \(K\) be a number field, \(f(x) = x^{q} + c \in K[x]\), where \(c\) is an algebraic integer, \(q = p^{r}\) is a power of a prime number \(p\) such that \(\nu_{\mathfrak{p}}(c) > 0\) for some prime \(\mathfrak{p}\) of \(K\) lying over \(p\) and suppose that \(0\) is not preperiodic for \(f\). Furthermore, suppose that \(\beta_{1}, \beta_{2}, \ldots, \beta_{t}\in K\) are finitely many strictly preperiodic elements for \(f\) such that for all \(i \neq j\), \(\beta_{i}\) does not lie in the forward orbit of \(\beta_{j}\) . Then, the system \((f, \{\beta_{i}\}_{i=1}^{t})\) forms an almost linearly-disjoint system.  
\label{thm:HanTucker}
\end{theorem}

\begin{proof}
As we discussed earlier, one may assume without loss of generality that there are no distinct \(j, k\) and positive \(m\) such that \(f^{m}(\beta_{j}) = \beta_{k}\). Then, \cite[Theorem 5.4]{HanTucker} states that for every \(1 \leq i \leq t\), 
\begin{equation} \label{eq:index}
L^{(i)} := K_{\infty}(f, \beta_{i}) \cap  \prod_{j \neq i} K_{\infty}(f, \beta_{j})
\end{equation}
is a finite extension of \(K\). If we define \(L := \prod_{i=1}^{t} L^{(i)}\), then for every \(i\in\{1, 2, \ldots, t\}\), we have \[L_{m}(f, \beta_{i}) \cap  \prod_{j \neq i} L_{m}(f, \beta_{j}) = L \text{ for every } m \geq 1.\]
\end{proof}

\begin{remark}
\begin{enumerate}
    \item Note that almost linear-disjointness, say of \(\big(\varphi, \{\alpha_{i}\}_{i=1}^{k}\big)\), is a milder requirement than the fact that \[K_{m}(\varphi, \alpha_{i}) \cap \prod_{j\neq i} K_{m}(\varphi, \alpha_{j}) \text{ be of finite index over } K \text{ for every } 1 \leq i \leq k.\] For instance, take \(\varphi(x) = x^{2}\) \(\alpha_{1} = 2, \alpha_{2} = 3\) and \(\alpha_{3} = 5\) over \(K = \mathbb{Q}\). In this case,
\[ \mathbb{Q}_{n}(\varphi, \alpha_{i}) = \mathbb{Q}(\zeta_{2^{n}}, \alpha_{i}^{1/2^{n}}), \text{ where } \zeta_{2^{n}} \text{ is the primitive root of unity of order } 2^{n}. \] We see that the intersections of any \(\mathbb{Q}_{n}(\varphi, \alpha_{i})\) with \(\prod_{j\neq i} \mathbb{Q}_{n}(\varphi, \alpha_{j})\) contains \(\mathbb{Q}(\zeta_{2^{n}})\) and hence becomes arbitrarily large over \(\mathbb{Q}\) although the system \(\{\varphi, \{2, 3, 5\}\) is almost linearly-disjoint.
    
    \item The condition that none of the \(\beta_{i}\) lie in each other's orbit, as supposed in Theorem \ref{thm:HanTucker}, is not essential for any of the results established in this article. We will only use this assumption as an intermediate step in establishing our main results. Eventually, we show in Section \ref{sec:augmentation} that all the results hold when orbits intersect too. 
\end{enumerate}

\end{remark}

\color{black}

\begin{remark}\label{rem:linear-disjoint2}
We will also use the following two classical facts regarding linear-disjointness of field extensions in this article. 
\begin{enumerate}
    \item Finitely many Galois extensions $M_{1}, M_{2}, \ldots, M_{n}$ of \(L\) are linearly disjoint over \(L\) if and only if \[ M_{i} \cap \prod_{j \neq i} M_{j} = L \text{  for every  } 1 \leq i \leq n. \]
    
    \item Suppose that we have finitely many Galois extensions $M_{1}, M_{2}, \ldots, M_{n}$ of a number field $L$ with respective Galois groups $G_{1}, G_{2}, \ldots, G_{n}$ and let $G$ be the Galois group of the compositum extension $\prod_{i=1}^{n} M_{i}$. It is well-known that if $M_{1}, M_{2}, \ldots, M_{n}$ are linearly disjoint over $L$, then we have \[ G \cong G_{1} \times G_{2} \times \cdots \times G_{n}. \]
\end{enumerate}
\end{remark}
For a finite group \(G\) acting on a finite set \(X\), we will say that an element \(g\in G\) is a \textit{derangement} of \(X\) if the set \(X^{g}\) of fixed points of \(g\) is empty.

\subsection{Density in Number Fields}
Given any subset $S$ of primes of \(K\), the Dirichlet density of $S$ is defined as 
\[ \delta_{K}(S) = \lim_{s\rightarrow 1^{+}} \frac{\sum_{\mathfrak{p}\in S} \frac{1}{(N\mathfrak{p})^{s}}}{\sum_{\mathfrak{p}} \frac{1}{(N\mathfrak{p})^{s}}} =  \lim_{s\rightarrow 1^{+}} \frac{\sum_{\mathfrak{p}\in S} \frac{1}{(N\mathfrak{p})^{s}} }{\log\big( \frac{1}{s-1}\big)} \text{ , if the limit exists. }\] Here, $N\mathfrak{p} = [\mathfrak{o}_{K}:\mathfrak{p}]$. Similarly, the natural density of $S$ is defined to be 
\[ d_{K}(S) = \lim_{N\rightarrow\infty} \frac{\lvert \mathfrak{p}\in S : N\mathfrak{p} \leq N \rvert }{ \lvert \mathfrak{p}  : N\mathfrak{p} \leq N \rvert} \text{ , if the limit exists. } \] It is well known that if $d_{K}(S)$ exists for a subset $S$ of primes of \(K\), then $\delta_{K}(S)$ also exists and $\delta_{K}(S) = d_{K}(S)$. The results in this note hold for both the notions of density above; so, we use \textit{density} without qualification. We will also employ the classical non-effective version of the Chebotarev Density Theorem; for more details, see, for example, \cite[Theorem 6.3.1]{FrJar} or \cite{SL96}.

\begin{theorem}\label{thm:chebotarev}
    Let $L/K$ be a Galois extension of number fields, $G$ be the Galois group of this extension and $\mathcal{C}$ be a conjugacy class in $G$. Then, the density of the subset of primes \(\mathfrak{p}\) of \(K\) such that \(\Big[ \frac{L/K}{\mathfrak{p}}\Big] = \mathcal{C}\) is equal to $\frac{\lvert \mathcal{C} \rvert}{\lvert G \rvert}$. 
\end{theorem}

\section{Some Preliminary Facts}\label{sec:prelim-result}
We will need the following lemma that roughly states that if \(\alpha\) is not periodic modulo a prime \(\mathfrak{p}\) of good reduction, then given a large enough extension \(L\), \(\big(\varphi^{m}(x) - \alpha\big)\) will fail to have any root in \(L_{\mathfrak{q}}\) for any large enough \(m\). This lemma is taken from \cite[Lemma 3.3]{BGHKST} and also gives additional information regarding existence of an extension \(E\) such that for every extension \(L\) of \(E\), there is an unramified prime \(\mathfrak{q}\) with the same property. 

\begin{lemma}\label{lem:nonper-modp}
Let \(K\) be a number field, let \(\tilde{\mathfrak{p}}\) be a prime of \(\mathfrak{o}_{\overline{K}}\), and let \(\varphi:\mathbb{P}^{1}\longrightarrow\mathbb{P}^{1}\) be a rational function defined over \(K\) and of good reduction at \(\mathfrak{p}=\tilde{\mathfrak{p}}\cap\mathfrak{o}_{K}\) such that \(2\leq\deg\varphi<\operatorname{char}k_{\mathfrak{p}}\). Suppose that \(\alpha\in\mathbb{P}^{1}(K)\) is not periodic modulo \(\mathfrak{p}\). Then there exists a finite extension \(E\) of \(K\) with the following property: for any finite extension \(L\) of \(E\), there is an integer \(M\in\mathbb{N}\) such that for all \(m\geq M\) and all \(\beta\in\mathbb{P}^{1}(\overline{K})\) with \(\varphi^{m}(\beta)=\alpha\),
\begin{itemize}
\item[(i)] \(\mathfrak{r}\) does not ramify over \(\mathfrak{q}\), and
\item[(ii)] \([\mathfrak{o}_{L(\beta)}/\mathfrak{r}:\mathfrak{o}_{L}/\mathfrak{q}]>1\),
\end{itemize}
where \(\mathfrak{r}:=\tilde{\mathfrak{p}}\cap\mathfrak{o}_{L(\beta)}\), and \(\mathfrak{q}:=\tilde{\mathfrak{p}}\cap\mathfrak{o}_{L}\).
\end{lemma}

As a consequence of Lemma \ref{lem:nonper-modp}, we obtain the following elementary combinatorial reformulation which will be essential in establishing our main result.

\begin{proposition}\label{prop:derangement-in-large-enough-extensions}
Let \(K\) be a number field, and let \(\varphi:\mathbb{P}^{1}\longrightarrow\mathbb{P}^{1}\) be a rational function defined over \(K\). Suppose that \(\alpha\in\mathbb{P}^{1}(K)\) is not \(\varphi\)-periodic. Then, there exists a finite extension \(E\) of \(K\) such that for any finite extension \(L\) of \(E\), there is a positive integer \(M\) such that for all \(m \geq M\), \(\mathrm{Gal}\big( L_{m}(\varphi, \alpha)/L\big)\) contains a conjugacy class of derangements on \(\varphi^{-m}(\alpha)\). 
\end{proposition}

\begin{proof}
    Since $\alpha$ is not \(\varphi\)-periodic, Corollary \ref{cor:nonper-red} gives us a positive density of primes \(\mathfrak{p}\) of \(\mathfrak{o}_{K}\) such that \(\varphi\) has good reduction at \(\mathfrak{p}\), \(\mathrm{char}(K_{\mathfrak{p}}) > \mathrm{deg}(\varphi)\) and \(\alpha (\operatorname{mod} \mathfrak{p})\) is not periodic under reduction \(\varphi_{\mathfrak{p}}\). Then, Lemma \ref{lem:nonper-modp} gives that there exists a finite extension \(E\) of K such that for any finite extension \(L\) of \(K\), there is a positive integer \(M\) such that for every \(m \geq M\) such that \(\varphi^{m}(x) - \alpha\) has no root in \(L\). Fix the \(E\) obtained as above from now on and also note that for any such \(m\), \(\varphi^{m}(x) - \alpha\) has bad reduction only at finitely many primes. So, we have that for every extension \(L\) of \(E\), there is a positive \(M\) such that for every \(m \geq M\), \(\varphi^{m}(x) - \alpha\) fails to have root modulo a positive density of primes \(\mathfrak{p}\) of good reduction. In other words, for every field extension \(L\) of \(E\), there exists a large enough \(m\) and a prime \(\mathfrak{q}\) in \(L\) such that the Artin conjugacy class \(\Big[ \frac{L_{m}(\varphi, \alpha)/L}{\mathfrak{q}} \Big]\) corresponding to \(\mathfrak{q}\) in the group \(\mathrm{Gal}\big(L_{m}(\varphi, \alpha)/L\big)\) does not fix any root of \(\big(\varphi^{m}(x) - \alpha\big)\), and hence this conjugacy class consists entirely of derangements on \(\varphi^{-m}(\alpha)\). 
\end{proof}

Now, we establish an analogue of Proposition \ref{prop:derangement-in-large-enough-extensions} to a finite set \(\mathcal{A} = \{\alpha_{i}\}_{i=1}^{k}\) of, none of which are \(\varphi\)-periodic. We will need an analogue for the field \(L_{m}(\varphi, \alpha)\) for a finite set \(\mathcal{A}\). Given a number field \(L\), \(\mathcal{A} = \{\alpha_{i}\}_{i=1}^{k} \subset\mathbb{P}^{1}(L)\) and a \(m \geq 1\), we define \(L_{m}(\varphi, \mathcal{A})\) to be the compositum \(\prod_{i=1}^{k} L_{m}\big(\varphi, \alpha_{i}\big)\).

\begin{proposition}\label{prop:derangement-finite-set}
    Let \(K\) be a number field, let \(\varphi:\mathbb{P}^{1}\longrightarrow\mathbb{P}^{1}\) be a rational function defined over \(K\) and let \( \mathcal{A} = \{\alpha_{i}\}_{i=1}^{k} \subset\mathbb{P}^{1}(K)\) be a finite subset, none of whose elements are \(\varphi\)-periodic modulo \(\mathfrak{p}\). Furthermore, assume that the set \( \{\alpha_{i}\}_{i=1}^{k}\) is almost \(\varphi\)-linearly disjoint. 
    
    Then, there is a finite extension \(E\) of \(K\) such that for every extension \(L\) of \(E\), there exists arbitrarily large positive \(m\), an extensions $L^{(m)}$ such that \(\mathrm{Gal}\Big((L^{(m)})_{m}(\varphi, \mathcal{A})/L^{(m)}\Big)\) contains a conjugacy class of derangements on \(\cup_{i=1}^{k} \big(\varphi^{-m}(\alpha_{i})\big)\). 
\end{proposition}

\begin{proof}
Since none of the \(\alpha_{i}\)'s are \(\varphi\)-periodic, Proposition \ref{prop:derangement-in-large-enough-extensions} gives for every \(1 \leq i \leq k\), the existence of extensions \(E_{i}\) such that for any finite extensions \(N_{i}\) of \(E_{i}\), there exists a positive integer \(M_{i}\) such that for all \(m \geq M_{i}\) the group \(\mathrm{Gal}\big((N_{i})_{m}(\varphi, \alpha_{i})/N_{i}\big)\) contains a conjugacy class of derangement of \(\varphi^{-m}(\alpha_{i})\). In this case, define \(E = \prod_{i=1}^{k} E_{i}\) to be the compositum of \(E_{i}\)'s. Then, the almost \(\varphi\)-linear disjointness of \(\{\alpha_{i}\}_{i=1}^{k}\) implies that there exists arbitrarily large \(m\) and extensions \(L^{(m)}\) of \(E\) such that the fields \[\big\{(L^{(m)})_{m}(\varphi, \alpha_{i}) \big\}_{i=1}^{k}\] are linearly disjoint over \(L^{(m)}\). 

Since \(L^{(m)}\) is an extension of \(E\), specifically of individual \(E_{i}\)'s, we have corresponding \(M_{i}\) such that for all \(m \geq M_{i}\), we have that \(\mathrm{Gal}\big((L^{(m)})_{m}(\varphi, \alpha_{i})/L^{(m)}\big)\) contains a conjugacy class \(\mathcal{C}_{i}\) of derangements on  \(\varphi^{-m}(\alpha_{i})\). For any \(m \geq \mathrm{max}\{M_{i}\}_{i=1}^{k}\) and the corresponding \(L^{(m)}\), we have that 
\begin{equation}\label{eq:directproduct-general}
   \mathrm{Gal}\Big((L^{(m)})_{m}(\varphi, \mathcal{A})/L^{(m)}\Big) \cong  \prod_{i=1}^{k} \mathrm{Gal}\Big((L^{(m)})_{m}(\varphi, \alpha_{i})/L^{(m)}\Big).
\end{equation}
Since each component on the RHS of \eqref{eq:directproduct-general} contains a conjugacy class \(\mathcal{C}_{i}\) such that every element \(\sigma_{i}\in\mathcal{C}_{i}\) acts a derangement of \(\varphi^{-m}(\alpha_{i})\), every element of the conjugacy class \[ \prod_{i=1}^{k} \mathcal{C}_{i} \subset \mathrm{Gal}\Big((L^{(m)})_{m}(\varphi, \mathcal{A})/L^{(m)}\Big) \] is derangement of the set \(\cup_{i=1}^{k}\big(\varphi^{-m}(\alpha_{i})\big)\). 
\end{proof}

To establish Theorem \ref{thm:main-intro} in its full generality, we will establish an analogue of Proposition \ref{prop:derangement-finite-set} for linearly disjoint system. Given a number field \(L\),  and finitely many subsets \(\mathcal{A}_{1}, \ldots, \mathcal{A}_{g}\) of \(\mathbb{P}^{1}(L)\) and a \(m \geq 1\), we define \(L_{m}(\{\varphi_{i}\}_{i=1}^{g}, \{\mathcal{A}_{i}\}_{i=1}^{g})\) to be the compositum \(\prod_{i=1}^{g} L_{m}(\varphi_{i}, \mathcal{A}_{i}).\)

\begin{proposition}\label{prop:derangement-finite-set-system}
   Let $K$ be a number field, \(g \in\mathbb{N}\) and \(\varphi_{1}, \ldots, \varphi_{g}: \mathbb{P}^{1}_{K} \longrightarrow \mathbb{P}^{1}_{K}\) be finitely many rational functions, each of degree at least $2$. For every \(1 \leq i \leq g\), let \(\mathcal{S}_{i} = \big\{\alpha_{ij} \big\}_{j=1}^{n_{i}}\subset\mathbb{P}^{1}(K)\) be a set consisting entirely of points that are not \(\varphi_{i}\)-periodic. Furthermore, suppose that the set \(\big\{(\varphi_{i}, \mathcal{S}_{i})\big\}_{i=1}^{g}\) forms an almost linearly-disjoint system. Then, there is a finite extension \(E\) of \(K\) such that for every extension \(N\) of \(E\), there exists arbitrarily large positive \(m\), and corresponding extensions $L = L^{(m)}$ such that \(\mathrm{Gal}\Big(L_{m}(\{\varphi_{i}\}_{i=1}^{g}, \{\mathcal{A}_{i}\}_{i=1}^{g})/L\Big)\) contains a conjugacy class of derangement of \(\bigcup_{i=1}^{k}\big( \bigcup_{j=1}^{n_{i}} \varphi^{-m}(\alpha_{ij})\big)\). 
\end{proposition}

\begin{proof}
Let \(1 \leq i \leq g\). Since none of the \(\alpha_{ij}\)'s are \(\varphi\)-periodic, Proposition \ref{prop:derangement-finite-set} gives a finite extension \(E_{i}\) of \(K\) such that for every extension \(N_{i}\) of \(E_{i}\), there exists arbitrarily large positive integer \(m\) and extensions \(L^{(m)}\) such that \(\mathrm{Gal}\Big((L^{(m)})_{m}(\varphi_{i}, \mathcal{S}_{i})/L^{(m)}  \Big)\) contains a conjugacy class of derangement of \(\cup_{i=1}^{n_{i}} \varphi^{-m}(\alpha_{ij})\). 

Define \(E = \prod_{i=1}^{g} E_{i}\) to be the compositum of \(E_{i}\)'s and \(N\) be a finite extension of \(E\). Since the system \(\big\{(\varphi_{i}, \mathcal{S}_{i})\big\}_{i=1}^{g}\) is almost linearly-disjoint, for every \(1 \leq i \leq g\), there exists arbitrarily large positive integer \(m\) and corresponding finite extensions \(L := L^{(m)}\) of \(N\) such that the the compositum fields \[ L_{m}(\varphi_{i}, \mathcal{S}_{i}) = \prod_{\alpha\in\mathcal{S}_{i}} L_{m}(\varphi, \alpha)  (1 \leq i \leq g) \] are linearly disjoint over \(L\). Now, choose a large enough positive \(m\) such that both the preceding statements hold and choose a corresponding $L^{(m)}$. Then, by the linear disjointness of the compositum fields, we have that 

\begin{equation}\label{eq:direct-product-system}
    \mathrm{Gal}\Big((L^{(m)})_{m}(\{\varphi_{i}\}_{i=1}^{g}, \{\mathcal{A}_{i}\}_{i=1}^{g})/L^{(m)}\Big) \cong \prod_{i=1}^{g} \mathrm{Gal}\Big((L^{(m)})_{m}(\varphi, \mathcal{S}_{i})/L^{(m)}\Big).
\end{equation}
Since for every \(1 \leq i \leq g\), \(\mathrm{Gal}\Big((L^{(m)})_{m}(\varphi_{i}, \mathcal{S}_{i})/L^{(m)}  \Big)\) contains a conjugacy class \(\mathcal{C}_{i}\) of derangements on \(\cup_{i=1}^{n_{i}} \varphi^{-m}(\alpha_{ij})\), \eqref{eq:direct-product-system} implies that the LHS contains a conjugacy class \(\mathcal{C}_{1} \times \ldots \times \mathcal{C}_{g}\) of derangements on \(\bigcup_{i=1}^{k}\big( \bigcup_{j=1}^{n_{i}} \varphi^{-m}(\alpha_{ij})\big)\).
\end{proof}

\section{Proof of Theorem \ref{thm:multimultiperiodic}, Theorem \ref{thm:multiperiodic} and Corollary \ref{cor:multiell-order}}\label{sec:Proofs-Main-Results}
\subsection{Proof of Theorem \ref{thm:multiperiodic} and \ref{thm:multimultiperiodic}}
Since Theorem \ref{thm:multiperiodic} is a special case of Theorem \ref{thm:multimultiperiodic} when \(g = 1\), we establish the latter. 
Proposition \ref{prop:derangement-finite-set-system} implies that there is an extension \(E\) of \(K\) such that for every extension \(L\) of \(E\), there exists arbitrarily large \(m\) and a corresponing field extension \(L^{(m)}\) of \(L\) such that the group \(\mathrm{Gal}\Big((L^{(m)})_{m}(\{\varphi_{i}\}_{i=1}^{g}, \{\mathcal{A}_{i}\}_{i=1}^{g})/L^{(m)}\Big)\) contains a conjugacy class \(\mathcal{C}\) of derangement of \(\bigcup_{i=1}^{k}\big( \bigcup_{j=1}^{n_{i}} \varphi^{-m}(\alpha_{ij})\big)\). Let \(m\) be any such arbitrarily large \(m\) and \(L := L^{(m)}\) be the corresponding extension. Then, the subset 
\[ \mathcal{S}_{m} := \Bigg\{ \mathfrak{q} : \Bigg[ \frac{(L^{(m)})_{m}(\{\varphi_{i}\}_{i=1}^{g}, \{\mathcal{A}_{i}\}_{i=1}^{g})/L^{(m)}}{\mathfrak{q}} \Bigg] = \mathcal{C}  \Bigg\} \] of primes of \(L^{(m)}\) has positive density as a consequence of Theorem \ref{thm:chebotarev}. Since \(\prod_{i=1}^{g} \prod_{j=1}^{n_{i}} \big(\varphi^{m}(x) - \alpha_{ij}\big)\) has bad-reduction only at finitely many primes in \(\mathcal{S}_{m}\), we can assume without loss of generality that \(\mathcal{S}_{m}\) does not have any prime of bad reduction. Since the Frobenius corresponding to \(\mathfrak{q}\in\mathcal{S}_{m}\) does not fix any element of \(\bigcup_{i=1}^{k}\big( \bigcup_{j=1}^{n_{i}} \varphi^{-m}(\alpha_{ij})\big)\), we have that \(\prod_{i=1}^{g} \prod_{j=1}^{n_{i}} \big(\varphi^{m}(x) - \alpha_{ij}\big)\) does not have root modulo \(\mathfrak{q}\) for any \(\mathfrak{q}\in\mathcal{S}_{m}\). Therefore, the subset \[ \mathcal{T}_{m} := \big\{ \mathfrak{q} \cap \mathfrak{o}_{K} : \mathfrak{q} \in \mathcal{S}_{m}  \big\}\] of primes of \(K\) also has positive density and \(\prod_{i=1}^{g} \prod_{j=1}^{n_{i}} \big(\varphi^{m}(x) - \alpha_{ij}\big)\) does not have root modulo \(\mathfrak{p}\) for every \(\mathfrak{p}\in\mathcal{T}_{m}\). Therefore, \(\alpha_{ij}\) are not periodic modulo \(\mathfrak{p}\) for a positive density of primes \(\mathfrak{p}\) for any \(1 \leq i \leq g\) and \(1 \leq j \leq n_{i}\). $\square$

\subsection{Proof of Corollary \ref{cor:multiell-order}} Since for every \(1 \leq i \leq g\), none of the points in the set \(\{Q_{ij}\}_{j=1}^{n_{i}}\) are torsion, same is true for the points \(\{[q^{n-1}](Q_{ij})\}_{j=1}^{n_{i}}\). Therefore, the set \(\{\alpha_{ij}\}_{j=1}^{n_{i}} = \big\{x\big([q^{n-1}](Q_{ij})\big)\big\}_{j=1}^{n_{i}}\) consists only of points that are not \(\varphi_{i}\)-preperiodic as a subset of $\mathbb{P}^{1}(K)$ for any \(1 \leq i \leq g\). Since \(x \circ [q] = \varphi \circ x\), and the system \(\big\{\big(\varphi_{i}, \{x(Q_{ij})\}_{j=1}^{n_{i}}\big)\big\}_{i=1}^{g}\) is assumed to be almost linearly-disjoint, Theorem \ref{thm:multimultiperiodic} gives that there exists a positive density of primes \(\mathfrak{p}\) of \(K\) such that for every \(1 \leq i \leq g\) and every \(1 \leq j \leq n_{i}\), \(\alpha_{ij}(\operatorname{mod} \mathfrak{p})\) is not \((\varphi_{i})_{\mathfrak{p}}\)-periodic. Therefore, using \(x \circ [q] = \varphi \circ x\), we have that for all \(m \geq 1\), \[ [q^{m+n-1}](Q_{ij})_{\mathfrak{p}} \neq [q^{n-1}](Q_{ij})_{\mathfrak{p}} \text{  for every } 1 \leq i \leq k \text{ and } 1 \leq j \leq n_{i}, \] which establishes that the \(q\)-primary part of the order of \((Q_{ij})_{\mathfrak{p}}\) must be divisible by $q^{r}$ for some $r \geq n$. Since we have a positive density of such primes \(\mathfrak{p}\) and each of the \(E_{i}\) has bad reduction only at finitely many primes of \(K\), the proof is complete. $\square$

\section{Proof of Theorem \ref{thm:main-intro}}\label{sec:proof-main-intro}
\subsection{Reduction to Finite Extensions} \label{subsec:finite-extension}
We note that it suffices to prove Theorem \ref{thm:main-intro} for any finite extension of \(K\). First of all, almost linear-disjointness of a system \(\big\{(\varphi_{i}, \mathcal{S}_{i})\big\}_{i=1}^{g}\) is preserved through finite extensions. Similarly, insolubility of \(\varphi_{i}^{m}(\gamma) \equiv \alpha (\operatorname{mod} \mathfrak{q})\) for a positive density of primes in a finite extension \(L\) of \(K\) implies its it insolubility modulo a positive density of primes in \(K\) too because given a positive density of primes \(\Sigma\) of an extension \(L\) of \(K\), the set \(\big\{ q \cap \mathfrak{o}_{K} : q \in \Sigma \big\}\) has positive density as subset of primes of \(K\). 

\subsection{Reduction to Fixed Points}
We can also assume that all the \(\varphi_{i}\) pre-periodic points are in fact \(\varphi_{i}\)-fixed-points. For instance, assume that \(\{\alpha_{i1}, \alpha_{i2}, \ldots, \alpha_{in_{i}} \}\subseteq\mathcal{A}_{i}\) are all the pre-periodic points under \(\varphi_{i}\) for any \(1 \leq i \leq g\). Then by definition, for every \(1 \leq j \leq n_{i}\), there exists $m_{j}\geq 0, n_{j}\geq 1$ such that \(\varphi_{i}^{m_{j}}(\alpha_{ij}) = \varphi_{i}^{m_{j}+n_{j}}(\alpha_{ij})\). In this case, we take \(a = \mathrm{max}\{m_{j}\}_{j=1}^{\nu}\) and \(b = \mathrm{lcm}\{n_{j}: 1 \leq j \leq \nu\}\). Now, we replace each \(\alpha\in\mathcal{A}_{i}\) by \(\varphi_{i}^{a}(\alpha)\) and enlarge the finite set \(\mathcal{T}_{i}\) by 
\[ \bigcup_{\gamma\in\mathcal{T}_{i}} \big\{ \varphi_{i}(\gamma), \varphi_{i}^{2}(\gamma), \ldots, \varphi_{i}^{b-1}(\gamma)\big\}.\] 

Then, we replace the map \(\varphi_{i}\) by \(\varphi_{i}^{b}\); notice that every \(\varphi_{i}\) pre-periodic point \(\alpha_{j}\) in \(\mathcal{A}_{i}\) is fixed by \(\varphi_{i}^{b}\). Assume that Theorem \ref{thm:main-intro} holds upon these modifications, then we note that at any prime \(\mathfrak{p}\) where every \(\varphi_{i}\) has good reduction, there exists \(m \geq M\) with \[ \varphi_{i}^{m}(\gamma) \equiv \alpha (\operatorname{mod} \mathfrak{p}) \text{ for the new data }\] implying that we have \[ \big( \varphi_{i}^{b} \big)^{m} \big( \gamma) \big)  \equiv \alpha (\operatorname{mod} \mathfrak{p}). \] But for new data, \(\gamma\in\mathcal{T}_{i}\) corresponds to some \(\varphi_{i}^{\ell}(\gamma^{\prime})\), where \(\gamma^{\prime}\) is from the original \(\mathcal{T}_{i}\) and \(1 \leq \ell \leq b-1\). Therefore, we have \[ \big( \varphi_{i}^{b} \big)^{m} \big( \varphi_{i}^{\ell}(\gamma^{\prime})) \big)  \equiv \alpha (\operatorname{mod} \mathfrak{p}). \] Similarly, \(\alpha\) in the new data is equal to \(\varphi_{i}^{a}(\alpha)\) in the old data which gives 
\[ \big( \varphi_{i}^{b} \big)^{m} \big( \varphi_{i}^{\ell}(\gamma^{\prime})) \big)  \equiv \varphi_{i}^{b}(\alpha) (\operatorname{mod} \mathfrak{p}) \] and hence \[ \varphi_{i}^{(b-1)m+ \ell}(\gamma^{\prime}) \equiv \alpha(\operatorname{mod} \mathfrak{p}) \] in the old data.

Upon these reductions, we will also need the following Proposition that appears in \cite{BGHKST} (see Proposition 1, pp. 642). 

\begin{proposition}\label{prop:same-ram-index}
Let \(K\) be a number field, \(\tilde{p}\) be a prime of \(\mathfrak{o}_{\overline{K}}\), \(\varphi: \mathbb{P}^{1} \longrightarrow \mathbb{P}^{1}\) be a rational function defined over \(K\) and of good reduction at \(\mathfrak{p} = \tilde{\mathfrak{p}} \cap \mathfrak{o}_{K}\) such that \(2 \leq \mathrm{deg}(\varphi) < \mathrm{char}(K_{\mathfrak{p}})\). Let \(\mathcal{A} = \{\alpha_{i}\}_{i=1}^{k}\) be a finite subset of \(\mathbb{P}^{1}(K)\) such that for every \(\alpha_{i}\in \mathcal{A}\) the following is satisfied:
\begin{itemize}
    \item If \(\alpha_{i}\) is not periodic, then \(\alpha_{i}\) is not periodic modulo \(\mathfrak{p}\). 

    \item If \(\alpha_{i}\) is periodic, then \(\alpha_{i}\) is a fixed point of \(\varphi\) and the ramification index of \(\varphi_{i}\) at \(\alpha_{i}\) over \(K\) is equal to the ramification index of \(\varphi_{i}\) at \(\alpha_{i}\) modulo \(\mathfrak{p}\). 
\end{itemize}
Then, there is a finite extension \(E\) of \(K\) such that for every finite extension \(L\) of \(E\), there exists a positive integer \(M\) such that for every \(m \geq M\) and all \(\beta\in\mathbb{P}^{1}(\overline{K})\) with \(\varphi^{m}(\beta)\in\mathcal{A}\) but \(\varphi^{t}(\beta)\not\in\mathcal{A}\) for every \(0 \leq t \leq m-1\) the following are satisfied:
\begin{itemize}
    \item The prime \(\mathfrak{r} := \tilde{\mathfrak{p}} \cap \mathfrak{o}_{L(\beta)}\) is not ramified over the prime \(\mathfrak{q} := \tilde{\mathfrak{p}} \cap \mathfrak{o}_{L}\). 

    \item The index \([\mathfrak{o}_{L(\beta)}/\mathfrak{r} : \mathfrak{o}_{L}/\mathfrak{q}] > 1\).
\end{itemize}
\end{proposition}

We note that there are only finitely many primes \(\mathfrak{p}\) of \(K\) such that any of the following are satisfied:
\begin{itemize}
    \item For some \(1 \leq i \leq g\), \(\varphi_{i}\) has bad reduction at \(\mathfrak{p}\). 

    \item The characteristic of \(K_{\mathfrak{p}}\)  is not larger than \(\mathrm{max}\{\mathrm{deg}(\varphi_{i}):1 \leq i \leq g\}\).

    \item There exists some \(1 \leq i \leq g\) and some point \(\alpha \in \mathcal{A}_{i}\) such that the ramification index of \(\varphi_{i}\) at \(\alpha_{i}\), modulo \(\mathfrak{p}\), is larger than the ramification index of \(\varphi_{i}\) at \(\alpha_{i}\) in \(K\). 
\end{itemize}
By reduction steps above, we have that for every \(1 \leq i \leq g\), every point in \(\mathcal{A}_{i}\setminus\mathcal{B}_{i}\) are fixed by \(\varphi_{i}\) and \(\mathcal{B}_{i}\) consists of all the points in \(\mathcal{A}_{i}\) that are not \(\varphi_{i}\)-preperiodic.  Then, Theorem \ref{thm:multimultiperiodic} implies that there is a positive density of primes \(\mathfrak{p}\) of \(K\) such that each of the sets \(\mathcal{A}_{i}\), \(1 \leq i \leq g\), satisfy the hypothesis in Proposition \ref{prop:same-ram-index} with the corresponding \(\varphi_{i}\). We choose one such prime \(\mathfrak{p}\) of \(K\) and a prime \(\tilde{p}\) of \(\mathfrak{o}_{\overline{K}}\) such that \(\mathfrak{p} = \tilde{\mathfrak{p}} \cap \mathfrak{o}_{K}\).

Now, applying Proposition \ref{prop:same-ram-index} for each of the \(\mathcal{A}_{i}\), \(\varphi_{i}\) and this same prime \(\tilde{\mathfrak{p}}\) gives that for every \(1 \leq i \leq g\), fields \(E_{i}\) that satisfy its conclusions for respective \(\varphi_{i}\). Define \(L := \prod_{i=1}^{k} E_{i}\). Then for every \(1 \leq i \leq g\) and all large enough \(m\) and all \(\beta\in\mathbb{P}^{1}(\overline{K})\) such that \(\varphi_{i}^{m}(\beta)\in\mathcal{A}_{i}\) but \(\varphi_{i}^{t}(\beta)\not\in\mathcal{A}_{i}\) for all \(0 \leq t \leq m-1\) satisfies that:
\begin{itemize}
    \item The prime \(\mathfrak{r} := \tilde{\mathfrak{p}} \cap \mathfrak{o}_{L(\beta)}\) is not ramified over the prime \(\mathfrak{q} := \tilde{\mathfrak{p}} \cap \mathfrak{o}_{L}\). 

    \item The index \([\mathfrak{o}_{L(\beta)}/\mathfrak{r} : \mathfrak{o}_{L}/\mathfrak{q}] > 1\).
\end{itemize}
By the same reduction argument as made at the start of this section, it suffices to prove the claim for \(L\). Let \(m\) be sufficiently large as guaranteed above and let \(M\) be the Galois extension of \(L\) obtained upon adjoining all the \(\beta\in\mathbb{P}^{1}(\overline{L})\) such that for some \(1 \leq i \leq g\), \(\varphi_{i}^{m}(\beta)\in\mathcal{A}_{i}\) but \(\varphi_{i}^{t}(\beta)\not\in\mathcal{A}_{i}\) for any \(0 \leq t < m\). Then, since \(\mathfrak{r}\) is unramified over \(\mathfrak{q}\), the extension \(M/L\) is unramified at \(\mathfrak{q}\) and since \([\mathfrak{o}_{L(\beta)}/\mathfrak{r} : \mathfrak{o}_{L}/\mathfrak{q}] > 1\) for every such \(\beta\), we have that \[ \prod_{i=1}^{k} \Big( \prod_{\alpha\in\mathcal{A}_{i}} \big(\varphi_{i}^{m}(x) - \alpha_{i} \big) \Big) \] does not have any root in \(\mathfrak{o}_{L}/\mathfrak{q}\). Therefore, the Galois group of \(M/L\) contains a conjugacy class \(\mathcal{C}\) worth of derangements. Then, the Chebotaev density theorem implies that the subset
\[ \mathcal{S} := \Big\{ \tau :  \Big[ \frac{M/L}{\tau} \Big] = \mathcal{C} \Big\} \]
of primes of \(L\) has positive density in \(L\) and for every \(\tau\), \[ \prod_{i=1}^{k} \Big( \prod_{\alpha\in\mathcal{A}_{i}} \big(\varphi_{i}^{m}(x) - \alpha_{i} \big) \Big) \] has no roots in \(\mathfrak{o}_{L}/\tau\). Now, we want to argue that the Theorem \ref{thm:main-intro} follows from the following claim that appears in \cite[pp. 644]{BGHKST}.

\textbf{Claim 1}: Let \(n \geq m\), \(z \in\mathbb{P}^{1}(L)\) \(\tau\in\mathcal{S}\) and \(1\leq i \leq g\). Assume that \(\varphi_{i}^{n}\) is congruent modulo \(\tau\) to some element of \(\mathcal{A}_{i}\). Then, \(\varphi_{i}^{t}\) modulo \(\tau\) is already congruent to some element of \(\mathcal{A}_{i}\) for some \(0\leq t \leq m-1\). 

Consider the set \(\mathcal{V}\) of primes \(\tau \in \mathcal{S}\) such that at least one of the following is satisfied;
\begin{enumerate}
    \item For some \(1 \leq i \leq g\), some \(\gamma\in\mathcal{T}_{i}\), some \(\alpha\in\mathcal{B}_{i}\) and some \(0 \leq t \leq n-1\), \(\varphi_{i}^{t}(\gamma) \equiv \alpha (\operatorname{mod} \tau)\). 

    \item For some \(1 \leq i \leq g\), some \(\alpha\in \mathcal{B}_{i}\) and \(\alpha^{\prime}\in\mathcal{A}_{i}\) and some \(1 \leq t \leq m\), \(\varphi_{i}^{t}(\alpha^{\prime}) \equiv \alpha (\operatorname{mod} \tau)\). 
\end{enumerate}
In \((1)\) above, since none of the \(\gamma\in\mathcal{T}_{i}\) are pre-periodic and each of the \(\alpha\) are fixed-points, we have that \(\varphi_{i}^{t}(\gamma) \neq \alpha_{ij}\) for any $t \geq 0$. Therefore, there are only finitely many primes in \(\mathcal{S}\) satisfying this condition. Again, since \(\alpha\) in the second condition are fixed points whereas \(\alpha^{\prime}\) are non preperiodic under \(\varphi_{i}\), the set of primes in \(\mathcal{S}\) that satisfy this condition is also finite. Therefore, the set \(\mathcal{V}\) is finite and the set \(\mathcal{U} := \mathcal{S}\setminus\mathcal{V}\) has positive density. We claim that the conclusions of Theorem \ref{thm:main-intro} hold for the set \(\mathcal{U}\) of primes, positive integer \(m\) and the field \(L\).

For the sake of contradiction, assume that there is an \(1 \leq i \leq g\), a \(\gamma\in\mathcal{T}_{i}\), an \(\alpha\in\mathcal{A}_{i}\), a \(\tau\in\mathcal{U}\) and a $n \geq m$ such that \(\varphi_{i}^{n}(\gamma) \equiv \alpha (\operatorname{mod} \tau)\). Then, by the Claim 1 above, there is a \(0 \leq t \leq m-1\) and \(\alpha^{\prime} \in\mathcal{A}_{i}\) such that \(\varphi_{i}^{t}(\gamma) \equiv \alpha^{\prime}(\operatorname{mod} \tau)\). Then, by the way we defined \(\mathcal{U}\), \(\alpha^{\prime}\) cannot be \(\varphi_{i}\)-periodic. Now, \(\varphi_{i}^{n}(\gamma) \equiv \alpha (\operatorname{mod} \tau)\) and \(\varphi_{i}^{t}(\gamma) \equiv \alpha^{\prime} (\operatorname{mod} \tau)\) together imply that \(\varphi_{1}^{n-t-1}\big(\varphi_{1}(\alpha^{\prime})\big) \equiv \alpha (\operatorname{mod} \tau)\). Since \(-t \geq -m + 1\), we have that \(m - t - 1 \geq 0\), which along with Claim 1, gives that there exists \(0 \leq k \leq m-1\) such that \(\varphi_{1}^{k+1}(\alpha^{\prime}) \equiv \alpha(\operatorname{mod} \tau)\) for some \(\alpha\in\mathcal{A}_{1}\), contradicting the item (2) above, used in defining \(\mathcal{V}\) and then \(\mathcal{U}\). $\square$

\section{Application to dynamical Mordell-Lang Problems}\label{sec:DML}

As an application of Theorem \ref{thm:main-intro}, we establish the Conjecture \ref{con:DMLC} over number fields, in some special cases as stated in Theorem \ref{thm:DML}. 

\subsection{Proof of Theorem \ref{thm:DML}}\label{proof:DML} 
If some $x_{i}$ is $\varphi_{i}$-preperiodic, we can absorb the first finitely many iterates that may lie on $V$ into trivial arithmetic progressions $\{nk + \ell\}_{n \geq 0}$ with $k = 0$. Thus, we may assume that $x_{i}$ is $\varphi_{i}$-periodic, of some period $j \geq 1$. By restricting our attention to progressions $\{nk + \ell\}_{n \geq 0}$ with $j \mid k$, it suffices to assume a \(\varphi_{i}\)-periodic $x_{i}$ is fixed by $\varphi_{i}$. So, we may therefore assume that no $x_{i}$ is $\varphi_{i}$-preperiodic. 

Since the set \(\big\{(\varphi_{i}, \mathcal{S}_{i})\big\}_{i=1}^{g}\) forms an almost linearly-disjoint system, where \(\mathcal{S}_{i}\) are the set of all the critical points of \(\varphi_{i}\), Theorem \ref{thm:main-intro} implies that there is a positive \(M\) and a set of primes \(\mathcal{P}\) of \(K\) of positive density such that for every \(1 \leq i \leq g\) and every \(\mathfrak{p}\in\mathcal{P}\), \(\varphi_{i}^{m}(x_{i})\) is not equivalent modulo \(\mathfrak{p}\) to any critical point of \(\varphi_{i}\), \(\mathrm{char}{K_{\mathfrak{p}}} > \mathrm{deg}(\varphi_{i})\) and \(\varphi_{i}\) has good reduction at \(\mathfrak{p}\). Fix any such
prime $\mathfrak{p}$, and note that the derivative of the reduction $\varphi_{i,\mathfrak{p}}$ is
nontrivial, because $\varphi_{i}$ has good reduction and
$1 \le \deg \varphi_{i} < \operatorname{char} k_{\mathfrak{p}}$. Thus, \(\varphi_{i}'\bigl(\varphi_{i}^{m}(x_{i})\bigr),\)
or its appropriate analogue if $\varphi_{i}^{m}(x_{i})$ lies in the residue class at
$\infty$, is a $p$-adic unit for all $m \ge M$. It follows that
$\varphi_{i}^{m}(x_{i}) \not\equiv \gamma \pmod{\mathfrak{p}}$ for any attracting periodic
point $\gamma$ of $\varphi_{i}$. Applying Proposition \ref{prop:orbitAP} yields the desired conclusion. \(\square\)

We will present a more special case of the dynamical Mordell-Lang conjecture in the next section that follows unconditionally. 

\section{Augmentation of Sets by Elements of Forward Orbits}\label{sec:augmentation}
In this section, we will show that as a consequence of methods and proofs in this article, it is relatively straightforward to enhance our result to the case when \((\{\varphi_{i}\}_{i=1}^{g}, \{\mathcal{S}_{i}\}_{i=1}^{g})\) forms an almost linearly-disjoint and we augment the sets \(\mathcal{S}_{i}\) by elements of \(\mathcal{O}^{+}(\alpha_{ij})\) for any \(\alpha_{ij}\in\mathcal{S}_{i}\). Note that upon this augmentation, the almost linear-disjointness of the system is destroyed. We first present an augmentation of Proposition \ref{prop:derangement-finite-set-system}.  

\begin{proposition}\label{prop:augmented-derangement-system}
    Let $K$ be a number field, \(g \in\mathbb{N}\) and \(\varphi_{1}, \ldots, \varphi_{g}: \mathbb{P}^{1}_{K} \longrightarrow \mathbb{P}^{1}_{K}\) be finitely many rational functions, each of degree at least $2$. For every \(1 \leq i \leq g\), let \(\mathcal{S}_{i} = \big\{\alpha_{ij} \big\}_{j=1}^{n_{i}}\subset\mathbb{P}^{1}(K)\) be a set consisting entirely of points that are not \(\varphi_{i}\)-periodic. Furthermore, suppose that the set \(\big\{(\varphi_{i}, \mathcal{S}_{i})\big\}_{i=1}^{g}\) forms an almost linearly-disjoint system. 
    
    Now, for every \(1 \leq i \leq g\) and every \(1 \leq j \leq n_{i}\), take finitely many natural numbers \(r_{ij}^{(1)}, \ldots, r_{ij}^{(m_{ij})}\) and define \(\mathcal{A}_{i} := \mathcal{S}_{i} \bigcup \big\{ \varphi_{i}^{r_{ij}^{(\nu)}}\big(\alpha_{ij}\big) : 1 \leq \nu \leq m_{ij}, 1 \leq j \leq n_{i} \big\}\) for every \(1 \leq i \leq g\). Then, there is a finite extension \(E\) of \(K\) such that for every extension \(L\) of \(E\), there exists arbitrarily large positive \(m\), and extensions $L := L^{(m)}$ such that \(\mathrm{Gal}\Big(L_{m}(\{\varphi_{i}\}_{i=1}^{g}, \{\mathcal{A}_{i}\}_{i=1}^{g}))/L^{(m)}\Big)\) contains a conjugacy class of derangement of \(\bigcup_{i=1}^{g}\big( \bigcup_{\alpha\in\mathcal{A}_{i}} \varphi_{i}^{-m}(\alpha)\big)\). 
\end{proposition}

\begin{proof}
We will show that the result follows straightforwardly from Proposition \ref{prop:derangement-finite-set-system}. For every \(1 \leq i \leq g\) and \(1 \leq j \leq n_{i}\), we choose the largest \(0 \leq \nu_{ij} \leq m_{ij}\) such that \(\varphi_{i}^{r_{ij}^{(\nu_{ij})}}(\alpha_{ij}) \in\mathcal{A}_{i}\). Then, for every \(1 \leq i \leq g\), define \(\mathcal{T}_{i}\) to be the set that contains each of these unique \(\varphi_{i}^{r_{ij}^{(\nu_{ij})}}(\alpha_{ij})\) as \(1 \leq j \leq n_{i}\).  Now, we note the elementary fact that for every \(1 \leq i \leq g\), \(1 \leq j \leq n_{i}\) and \(1 \leq \nu \leq m_{ij}\), every root of \(\varphi^{m}(x) - \alpha_{ij}\) is also a root of \(\varphi^{m+r_{ij}^{(\nu)}}(x) - \varphi_{i}^{r_{ij}^{(\nu)}}(\alpha_{ij})\)  and hence for every field \(L\), we have 
\begin{equation}\label{eq:field-containment}
L_{m}(\varphi_{i}, \mathcal{A}_{i}) \subseteq L_{m+r_{ij}^{(\nu_{ij})}}(\varphi_{i}, \mathcal{T}_{i}).
\end{equation}
Since \(\big\{(\varphi_{i}, \mathcal{S}_{i})\big\}_{i=1}^{g}\) forms an almost linearly-disjoint system, we also have that \(\big\{(\varphi_{i}, \mathcal{T}_{i})\big\}_{i=1}^{g}\) forms an almost linearly-disjoint system by \eqref{eq:field-containment}. Now, we apply Proposition \ref{prop:derangement-finite-set-system} for the maps \(\varphi_{i}\) and sets \(\mathcal{T}_{i}\) gives the required conclusion for maps \(\{\varphi_{i}\}_{i=1}^{g}\) and the sets \(\{\mathcal{T}_{i}\}_{i=1}^{g}\). In particular, for every finite extension \(E\) of \(K\), there exists arbitrarily large \(m\) and corresponding extensions \(L := L^{(m)}\) such that \(\mathrm{Gal}\Big(L_{m}(\{\varphi_{i}\}_{i=1}^{g}, \{\mathcal{T}_{i}\}_{i=1}^{g}))/L^{(m)}\Big)\) contains a conjugacy class of derangement of \[\bigcup_{i=1}^{g}\big( \bigcup_{\alpha\in\mathcal{T}_{i}} \varphi_{i}^{-m}(\alpha)\big) = \bigcup_{i=1}^{g} \bigg( \bigcup_{j=1}^{n_{i}}  \varphi_{i}^{-m}\big( \varphi^{r_{ij}^{(\nu_{ij})}}(\alpha_{ij})\big) \bigg).\] Now we have the result since every derangement of the above set is also a derangement of \(\bigcup_{i=1}^{g}\big( \bigcup_{\alpha\in\mathcal{A}_{i}} \varphi_{i}^{-m}(\alpha)\big)\), since every root of \(\varphi^{m}(x) - \alpha_{ij}\) is also a root of \(\varphi^{m+r}(x) - \varphi_{i}^{r}(\alpha_{ij})\) for every \(m, r\) and \(i, j\). 
\end{proof}

Then, Theorem \ref{thm:multimultiperiodic} and Theorem \ref{thm:main-intro} both admit a similar enhancement. Let us first state the same for Theorem \ref{thm:multimultiperiodic}. 

\begin{theorem}\label{thm:augmented-multimultiperiodic}
    Let $K$ be a number field, \(g \in\mathbb{N}\) and \(\varphi_{1}, \ldots, \varphi_{g}: \mathbb{P}^{1}_{K} \longrightarrow \mathbb{P}^{1}_{K}\) be finitely many rational functions, each of degree at least $2$. For every \(1 \leq i \leq g\), let \(\mathcal{S}_{i} = \big\{\alpha_{ij} \big\}_{j=1}^{n_{i}}\subset\mathbb{P}^{1}(K)\) be a set consisting entirely of points that are not \(\varphi_{i}\)-periodic. Furthermore, suppose that the set \(\big\{(\varphi_{i}, \mathcal{S}_{i})\big\}_{i=1}^{g}\) forms an almost linearly-disjoint system. 

    Now, for every \(1 \leq i \leq g\) and every \(1 \leq j \leq n_{i}\), take finitely many natural numbers \(r_{ij}^{(1)}, \ldots, r_{ij}^{(m_{ij})}\) and define \(\mathcal{A}_{i} := \mathcal{S}_{i} \bigcup \big\{ \varphi_{i}^{r_{ij}^{(\nu)}}\big(\alpha_{ij}\big) : 1 \leq \nu \leq m_{ij}, 1 \leq j \leq n_{i} \big\}\) for every \(1 \leq i \leq g\). Then, there exists a positive density subset of primes $\mathfrak{p}$ of $K$ at which each of the \(\varphi_{i}\) has good reduction such that for every \(1 \leq i \leq g\), every \(\alpha\in\mathcal{A}_{i}\), \(\alpha\) is not \(\varphi_{i}\)-periodic modulo \(\mathfrak{p}\).
\end{theorem}

\begin{proof}
    The proof is exactly the same as the proof of Theorem \ref{thm:multimultiperiodic} that appears in Section \ref{sec:Proofs-Main-Results}, except that one needs to employ Proposition \ref{prop:augmented-derangement-system} where Proposition \ref{prop:derangement-finite-set-system} was used.
\end{proof}
\begin{remark}\label{rem:exceptional-point}
    One can easily see the above density result also holds if we replace the sets \(\mathcal{A}_{i}\) in Theorem \ref{thm:augmented-multimultiperiodic} by \((\mathcal{A}_{i}\cup E_{i})\), where \(E_{i}\) is the non-periodic exceptional points of \(\varphi_{i}\). This is because a non-periodic exceptional point \(\alpha\) is pre-periodic, and hence \(\{\varphi^{m}(\alpha) - \alpha\}_{m\geq 1}\) is a finite set that is divisible by only finitely many primes of \(K\).
\end{remark}
Now, we state the analogous augmentation for Theorem \ref{thm:main-intro}, whose proof is also exactly the same as in Section \ref{sec:proof-main-intro}, except that Theorem \ref{thm:augmented-multimultiperiodic} in place of Theorem \ref{thm:multimultiperiodic}. 

\begin{theorem}\label{thm:augmented-main-intro}
    Let \(K\) be a number field, let \(\varphi_{1},\ldots,\varphi_{g}:\mathbb{P}_{K}^{1}\longrightarrow\mathbb{P}_{K}^ {1}\) be rational maps of degree at least 2, and let \({\mathcal{A}}_{1},\ldots,{\mathcal{A}}_{g}\) be finite subsets of \(\mathbb{P}^{1}(K)\). Let \(\mathcal{B}_{i} \subseteq \mathcal{A}_{i}\) be the subset consisting of all the points in \(\mathcal{A}_{i}\) that are not \(\varphi_{i}\)-preperiodic, if there are any. Furthermore, suppose that the set \(\big\{(\varphi_{i}, \mathcal{B}_{i})\big\}_{i=1}^{g}\) forms an almost linearly-disjoint system. 

    Now, for every \(1 \leq i \leq g\) and every \(\alpha\in\mathcal{B}_{i}\), choose finitely many positive integers \(r_{1}, r_{2}, \ldots, r_{\nu}\) (which can vary for different \(i\) and \(\alpha\)) and add all the elements \(\varphi^{r_{1}}(\alpha), \varphi^{r_{2}}(\alpha), \ldots, \varphi^{r_{\nu}}(\alpha)\) to \(\mathcal{A}_{i}\). 
    Let \({\mathcal{T}}_{1},\ldots,{\mathcal{T}}_{g}\) be finite subsets of \(\mathbb{P}^{1}(K)\) such that no \({\mathcal{T}}_{i}\) contains any \(\varphi_{i}\)-preperiodic points. Then, there is a positive integer \(M\) and a set of primes \({\mathcal{P}}\) of \(K\) with positive density such that for any \(i=1,\ldots,g\), any \(\gamma\in{\mathcal{T}}_{i}\), any \(\alpha\in{\mathcal{A}}_{i}\), any \(\mathfrak{p}\in{\mathcal{P}}\), and any \(m\geq M\), \[\varphi_{i}^{m}(\gamma)\not\equiv\alpha\pmod{\mathfrak{p}}.\]
\end{theorem}

Finally, such an augmentation can also be done on critical points of functions, in the context of the dynamical Mordell-Lang conjecture. 

\begin{theorem}\label{thm:augmented-DML}
Let $K$ be a number field, \(g \in\mathbb{N}\), \(V \subset (\mathbb{P}^{1})^{g}\) be a subvariety defined over \(K\), let \(x = (x_{1}, \ldots, x_{g})\in(\mathbb{P}^{1})^{g}(K)\). Moreover, let \(\Phi := (\varphi_{1}, \ldots, \varphi_{g})\) act on \((\mathbb{P}^{1})^{g}\) coordinatewise, where each \(\varphi_{i}\in K(t)\) is a rational function of degree \(\geq 2\). Let \(\mathcal{T}_{i}\) be the set of critical points of \(\varphi_{i}\) and assume that for every \(1 \leq i \leq g\), there exists a subset \(\mathcal{S}_{i}\subseteq\mathcal{T}_{i}\) such that the set \(\big\{(\varphi_{i}, \mathcal{S}_{i})\big\}_{i=1}^{g}\) forms an almost linearly-disjoint system and every element in \(\mathcal{T}_{i}\setminus\mathcal{S}_{i}\) lies in the forward orbit of some element in \(\mathcal{S}_{i}\). Then, the set of integers $n \in \mathbb{N}$ such that $\Phi^{n}(x) \in V(\overline{K})$ is a union of finitely many arithmetic progressions $\{nk + \ell\}_{n \in \mathbb{N}}$, where $k,\ell \ge 0$ are nonnegative integers.  
\end{theorem}

\begin{proof}
    The proof is exactly analogous to the proof of Theorem \ref{thm:DML} in Section \ref{proof:DML} except that we employ Theorem \ref{thm:augmented-main-intro} where we previously used Theorem \ref{thm:main-intro}. 
\end{proof}

The following very specific case of the dynamical Mordell-Lang conjecture that follows unconditionally is worth stating. 

\begin{theorem}\label{thm:bicritical}
   Let $K$ be a number field, \(g \in\mathbb{N}\), \(V \subset (\mathbb{P}^{1})^{g}\) be a subvariety defined over \(K\), let \(x = (x_{1}, \ldots, x_{g})\in(\mathbb{P}^{1})^{g}(K)\). Moreover, let \(\Phi := (f, f, \ldots, f)\) act on \((\mathbb{P}^{1})^{g}\) coordinatewise, where \(f\in K[t]\) is a polynomial, all of whose critical points lie in the forward orbit of a single critical point. Then, the set of integers $n \in \mathbb{N}$ such that $\Phi^{n}(x) \in V(\overline{K})$ is a union of finitely many arithmetic progressions $\{nk + \ell\}_{n \in \mathbb{N}}$, where $k,\ell \ge 0$ are nonnegative integers.   
\end{theorem}

\section{Implications for Arboreal Galois Representations}\label{sec:Arboreal}
Let  \(\varphi : \mathbb{P}^{1}_{K} \longrightarrow \mathbb{P}^{1}_{K}\) be a rational map of degree \(d \geq 2\), \(\alpha\in\mathbb{P}^{1}_{K}\) and $\varphi^{-n}(\alpha)$ be the set of roots of $\varphi^{n}(x) - \alpha$ for \(n \geq 0\). One can define a graph on the set of vertices $T_{\varphi} = \bigsqcup_{n\geq 0} \varphi^{-n} (\alpha)$, with the rule that a vertex $B$ is a child of another vertex $A$ if and only if $\varphi(B) = A$. This graph is a complete rooted $d$-ary tree $T_{\varphi}$ if \(\alpha\) is not in the post-critical set of \(\varphi\), which we assume implicitly throughout this section. Then, the Galois group of the algebraic closure $\overline{K}/K$ acts on and preserves the tree-structure of $T_{\varphi}$. Therefore, one can define a morphism \[ \Phi_{\varphi} : \mathrm{Gal}\big(\overline{K}/K\big) \longrightarrow \mathrm{Aut}(T_{\varphi}) \] that is known as the \textit{arboreal Galois representation associated with \(\varphi\) and \(\alpha\)}. The image of this arboreal Galois representation, denoted by $G_{\infty}(\varphi, \alpha)$ can be seen as inverse limits of Galois groups \[G_{\infty}(\varphi, \alpha) := \mathrm{Gal}\big(K(\varphi^{-n}(\alpha))\big)/K.\] The study of arboreal Galois representations goes back to Odoni and Stoll in \cite{Odoni85, Odoni88, Odoni97, OdoniIterates, Stoll92} and readers can consult \cite{JonesSurvey} for a comprehensive survey regarding arboreal representations. 

The results in this article give rise to a pleasant interpretation and provides further insights in terms of \textit{multi-trees}, which are generalization of the \(d\)-ary rooted tree that are defined in the context of \textit{arboreal Galois representations}. Multi-tree interpretation of related results also appears in \cite{HanTucker} and we follow analogous definitions. First, we define these multi-trees  that arise from a single rational map and multiple elements. Suppose \(\varphi\) is of degree \(d \geq 2\) and \(\mathcal{A} = \{\alpha_{i}\}_{i=1}^{k}\in\mathbb{P}^{1}(K)\). For every \(n \geq 1\), let us define \[ \mathcal{M}_{n}(\varphi, \mathcal{A}) := \bigcup_{i=1}^{k} \Big( \bigcup_{i=0}^{n} \varphi^{-i}(\alpha_{i}) \Big) \quad \text{ and } \mathcal{G}_{n}(\varphi, \mathcal{A}) := \mathrm{Gal}\Big(\mathcal{M}_{n}(\varphi, \mathcal{A}) / K\Big). \] \(\mathcal{M}_{n}(\varphi, \mathcal{A}))\) can be naturally seen as the disjoint-union of \(k\) different trees, rooted at each \(\alpha_{i}\). Then, we analogously define \( \mathcal{M}_{\infty}(\varphi, \mathcal{A}) \) and \(\mathcal{G}_{\infty}(\varphi, \mathcal{A})\) respectively as the direct limit of \(\mathcal{M}_{n}(\varphi, \mathcal{A})\) and inverse limit of \(\mathcal{G}_{n}(\varphi, \mathcal{A})\), just as in the single-element case. Then, since each \(\mathcal{G}_{n}(\varphi, \mathcal{A})\) acts faithfully on each \(\mathcal{M}_{n}(\varphi, \mathcal{A})\), one has injections \(\mathcal{G}_{n}(\varphi, \mathcal{A}) \hookrightarrow \mathrm{Aut}\big(\mathcal{M}_{n}(\varphi, \mathcal{A})\big) \) for every \(n\) and hence an injection \(\mathcal{G}_{\infty}(\varphi, \mathcal{A}) \hookrightarrow \mathrm{Aut}\big(\mathcal{M}_{\infty}(\varphi, \mathcal{A})\big) \).  Then, Proposition \ref{prop:derangement-finite-set} has an interpretation as follows:

\begin{proposition}\label{prop:derangements-multitree}
    Let \(K\) be a number field, let \(\varphi:\mathbb{P}^{1}_{K}\longrightarrow\mathbb{P}^{1}_{K}\) be a rational function of degree at least 2, and let \(\mathcal{A} = \{\alpha_{i}\}_{i=1}^{k}\subset\mathbb{P}^{1}(K)\) be a subset consisting entirely of non-periodic points. Furthermore, assume that the set \(\mathcal{A}\) is almost \(\varphi\)-linearly disjoint. Then, there exists a positive \(n\) such that for every \(m \geq n\), the groups \(\mathcal{G}_{m}(\varphi, \mathcal{A})\) consists of a conjugacy class worth of derangements on the \(m^{th}\)-level \(\mathcal{M}_{m}(\varphi, \mathcal{A})\) of the infinite multi-tree \(\mathcal{M}_{\infty}(\varphi, \mathcal{A})\).
\end{proposition}

A similar result also follows analogously from Proposition \ref{prop:derangement-finite-set-system} in the context of what we call \textit{aboreal forest} associated to multiple rational maps and multiple points. Let \(\varphi_{1}, \ldots, \varphi_{g}: \mathbb{P}^{1}_{K} \longrightarrow \mathbb{P}^{1}_{K}\) be finitely many rational maps, each of degree at least $2$ and \(\mathcal{A}_{i} = \big\{\alpha_{ij} \big\}_{j=1}^{n_{i}}\subset\mathbb{P}^{1}(K)\) be finite subsets consisting entirely of points that are not \(\varphi_{i}\)-periodic, for every \(1 \leq i \leq g\). We define the \(n^{th}\)-level, \(\mathcal{F}_{n}\big(\{\varphi_{i}\}_{i=1}^{g}, \{\mathcal{A}_{i}\}_{i=1}^{g}\big)\), of the \textit{arboreal forest} associated to these given maps and sets as the disjoint union \[\bigsqcup_{i=1}^{g} \mathcal{M}_{n}(\varphi_{i}, \mathcal{A}_{i}) \] of multi-trees associated to respective \(\varphi_{i}\) and \(\mathcal{A}_{i}\). Then, we analogously define the infinite arboreal forest \(\mathcal{F}_{\infty}\big(\{\varphi_{i}\}_{i=1}^{g}, \{\mathcal{A}_{i}\}_{i=1}^{g}\big)\) as the direct limit of \(\mathcal{F}_{n}\big(\{\varphi_{i}\}_{i=1}^{g}, \{\mathcal{A}_{i}\}_{i=1}^{g}\big)\).  Since we have an injection \(\mathcal{G}_{\infty}(\varphi_{i}, \mathcal{A}_{i}) \hookrightarrow \mathrm{Aut}\big(\mathcal{M}_{\infty}(\varphi_{i}, \mathcal{A}_{i})\big) \) for every \(1 \leq i \leq g\), if \(\big\{(\varphi_{i}, \mathcal{A}_{i})\big\}_{i=1}^{g}\) forms an almost linearly-disjoint system then we also have injection \[ \prod_{i=1}^{g} \mathcal{G}_{n}(\varphi_{i}, \mathcal{A}_{i}) \hookrightarrow \prod_{i=1}^{g} \mathrm{Aut}\big(\mathcal{M}_{n}(\varphi_{i}, \mathcal{A}_{i})\big) \text{ for arbitrarily large } n. \] 

\begin{proposition}
    Let $K$ be a number field, \(g \in\mathbb{N}\) and \(\varphi_{1}, \ldots, \varphi_{g}: \mathbb{P}^{1}_{K} \longrightarrow \mathbb{P}^{1}_{K}\) be finitely many rational functions, each of degree at least $2$. For every \(1 \leq i \leq g\), let \(\mathcal{S}_{i} = \big\{\alpha_{ij} \big\}_{j=1}^{n_{i}}\subset\mathbb{P}^{1}(K)\) be a set consisting entirely of points that are not \(\varphi_{i}\)-periodic. Furthermore, suppose that the set \(\big\{(\varphi_{i}, \mathcal{S}_{i})\big\}_{i=1}^{g}\) forms an almost linearly-disjoint system. Then, there exists a positive \(n\) such that for every \(m \geq n\), the groups \(\mathcal{G}_{m}(\varphi_{i}, \mathcal{A}_{i})\) contains a conjugacy class of derangements on the \(m^{th}\) level \(\mathcal{F}_{m}\big(\{\varphi_{i}\}_{i=1}^{g}, \{\mathcal{A}_{i}\}_{i=1}^{g}\big)\) of the infinite arboreal forest \(\mathcal{F}_{\infty}\big(\{\varphi_{i}\}_{i=1}^{g}, \{\mathcal{A}_{i}\}_{i=1}^{g}\big)\)
\end{proposition}

\section{Future Directions}\label{sec:future}
There are non-periodic elements \(\alpha_{1}, \alpha_{2}\) for which the methods in this article are currently not sufficient to establish non-periodicity of both \(\alpha_{1}, \alpha_{2}\) modulo a positive density of primes. For example, if we take \(\varphi(x) = x^{2}\) and \(\alpha_{1} = 2, \alpha_{2} = 8\) over \(K = \mathbb{Q}\). Then, for every \(m \geq 1\), we have \[\mathbb{Q}_{m}(\varphi, 8) = \mathbb{Q}(\mu_{2^{m}}, 8^{1/2^{m}}) \subseteq \mathbb{Q}(\mu_{2^{m}}, 2^{1/2^{m}}) = \mathbb{Q}_{m}(\varphi, 2).\] In the language of derangements, in such cases, one would need to establish an analogue of Proposition \ref{prop:derangement-finite-set-system} that \(\mathrm{Gal}\big(L_{m}(\varphi, 2)/L_{m}\big)\) contains derangements of \(\varphi^{-m}(2) \cup \varphi^{-m}(8)\), for arbitrary large \(m\). Of course, in this specific case, an elementary examination of Galois groups yields the desired result. In particular, \textit{power maps}, \textit{Chebyshev maps} and \textit{Latt\'es maps} constitute somewhat \textit{special} cases in this regard. These maps arise as quotients of endomorphisms of one dimensional connected algebraic groups (see \cite[Chapter 6]{SilBook}) and that they are the only rational maps that have multipliers contained within a given number field (see \cite{Huguin}). 

The assumption of almost linear disjointness used in this article should be viewed as an \textit{idealized} expectation that apart from above special maps, \(L_{m}(\varphi, \alpha_{1}) \subseteq L_{m}(\varphi, \alpha_{2})\) can hold for all large enough \(m\) and infinitely many extensions \(L\) of \(K\) only because of a dynamical reason. In particular, either one of the \(\alpha_{1}, \alpha_{2}\) is exceptional or one lies in the forward orbit of the other. And, we have demonstrated in Section \ref{sec:augmentation} that all the results in this article can be obtained when the sets associated with an almost linearly-disjoint system are augmented in both the conditions. 

\singlespacing
\section*{Acknowledgements}
The author is very grateful to Thomas J. Tucker for reading an earlier draft of some results in this article and directing him to \cite{Fak14}, to Rafe Jones for an illuminating conversation regarding arboreal representations and directing him to \cite{HanTucker}, and to Robert L. Benedetto for an interesting conversation. The author is also grateful to the anonymous reviewer for several suggestions that enhanced the overall presentation of this article and for directing him to \cite{DMLBook, Xie17}.

\bibliographystyle{amsplain}
\bibliography{local_global}

@article{Perucca,
 author = {Perucca, Antonella},
 title = {Prescribing valuations of the order of a point in the reductions of abelian varieties and tori},
 fjournal = {Journal of Number Theory},
 journal = {J. Number Theory},
 issn = {0022-314X},
 volume = {129},
 number = {2},
 pages = {469--476},
 year = {2009},
 language = {English},
 doi = {10.1016/j.jnt.2008.07.004},
 zbMATH = {5505561},
 Zbl = {1166.14028}
}

@book{SilBook,
 author = {Silverman, Joseph H.},
 title = {The arithmetic of dynamical systems},
 fseries = {Graduate Texts in Mathematics},
 series = {Grad. Texts Math.},
 issn = {0072-5285},
 volume = {241},
 isbn = {978-0-387-69903-5},
 year = {2007},
 publisher = {New York, NY: Springer},
 language = {English},
 zbMATH = {5134076},
 Zbl = {1130.37001}
}

@article{Huguin,
 author = {Huguin, Valentin},
 title = {Rational maps with rational multipliers},
 fjournal = {Journal de l'{\'E}cole Polytechnique -- Math{\'e}matiques},
 journal = {J. {\'E}c. Polytech., Math.},
 issn = {2429-7100},
 volume = {10},
 pages = {591--599},
 year = {2023},
 language = {English},
 doi = {10.5802/jep.227},
 zbMATH = {7680768},
 Zbl = {1517.37094}
}

@article{Odoni85,
 author = {Odoni, R. W. K.},
 title = {On the prime divisors of the sequence {{\(w_{n+1}=1+w_ 1{{\dots}} w_ n\)}}},
 fjournal = {Journal of the London Mathematical Society. Second Series},
 journal = {J. Lond. Math. Soc., II. Ser.},
 issn = {0024-6107},
 volume = {32},
 pages = {1--11},
 year = {1985},
 doi = {10.1112/jlms/s2-32.1.1},
 zbMATH = {3916364},
 Zbl = {0574.10020}
}

@article{Odoni88,
 author = {Odoni, R. W. K.},
 title = {Realising wreath products of cyclic groups as {Galois} groups},
 fjournal = {Mathematika},
 journal = {Mathematika},
 issn = {0025-5793},
 volume = {35},
 number = {1},
 pages = {101--113},
 year = {1988},
 doi = {10.1112/S002557930000632X},
 zbMATH = {4081685},
 Zbl = {0662.12010}
}

@article{Odoni97,
 author = {Odoni, R. W. K.},
 title = {On the {Galois} groups of iterated generic additive polynomials},
 fjournal = {Mathematical Proceedings of the Cambridge Philosophical Society},
 journal = {Math. Proc. Camb. Philos. Soc.},
 issn = {0305-0041},
 volume = {121},
 number = {1},
 pages = {1--6},
 year = {1997},
 doi = {10.1017/S0305004196001168},
 zbMATH = {1066712},
 Zbl = {0917.12001}
}

@article{Juul,
 author = {Juul, Jamie},
 title = {Backward orbits of critical points},
 fjournal = {Journal of Number Theory},
 journal = {J. Number Theory},
 issn = {0022-314X},
 volume = {256},
 pages = {23--36},
 year = {2024},
 doi = {10.1016/j.jnt.2023.09.006},
 zbMATH = {7794603},
 Zbl = {1535.37112}
}

@incollection{JonesSurvey,
 author = {Jones, Rafe},
 title = {Galois representations from pre-image trees: an arboreal survey},
 booktitle = {Actes de la conf\'erence ``Th\'eorie des nombres et applications''},
 isbn = {978-2-84867-472-8},
 pages = {107--136},
 year = {2013},
 publisher = {Besan{\c{c}}on: Presses Universitaires de Franche-Comt{\'e}},
 url = {pmb.univ-fcomte.fr/2013/Jones.pdf},
 zbMATH = {6308079},
 Zbl = {1307.11069}
}

@misc{HanTucker,
 author = {Han, Minsik and Tucker, Thomas J.},
 title = {Finite index theorems for iterated {Galois} groups of preperiodic points for unicritical polynomials},
 year = {2025},
 howpublished = {Preprint, {arXiv}:2508.00266 [math.{NT}] (2025)},
 url = {https://arxiv.org/abs/2508.00266},
 arXiv = {arXiv:2508.00266}
}

@article{BGKT12,
 author = {Benedetto, Robert L. and Ghioca, Dragos and Kurlberg, P{\"a}r and Tucker, Thomas J.},
 title = {A case of the dynamical {Mordell}-{Lang} conjecture, With an appendix by {Umberto Zannier}},
 fjournal = {Mathematische Annalen},
 journal = {Math. Ann.},
 issn = {0025-5831},
 volume = {352},
 number = {1},
 pages = {1--26},
 year = {2012},
 doi = {10.1007/s00208-010-0621-4},
 zbMATH = {6005026},
 Zbl = {1285.37021}
}

@article{BGHKST,
 author = {Benedetto, Robert L. and Ghioca, Dragos and Hutz, Benjamin and Kurlberg, P{\"a}r and Scanlon, Thomas and Tucker, Thomas J.},
 title = {Periods of rational maps modulo primes},
 fjournal = {Mathematische Annalen},
 journal = {Math. Ann.},
 issn = {0025-5831},
 volume = {355},
 number = {2},
 pages = {637--660},
 year = {2013},
 doi = {10.1007/s00208-012-0799-8},
 zbMATH = {6133902},
 Zbl = {1317.37111}
}

@article{Fak14,
 author = {Fakhruddin, Najmuddin},
 title = {The algebraic dynamics of generic endomorphisms of {{\(\mathbb{P}^n\)}}},
 fjournal = {Algebra \& Number Theory},
 journal = {Algebra Number Theory},
 issn = {1937-0652},
 volume = {8},
 number = {3},
 pages = {587--608},
 year = {2014},
 doi = {10.2140/ant.2014.8.587},
 zbMATH = {6322089},
 Zbl = {1317.37116}
}

@book{FrJar,
 author = {Fried, M. D. and Jarden, M.},
 title = {Field arithmetic},
 edition = {2nd revised and enlarged ed.},
 fseries = {Ergebnisse der Mathematik und ihrer Grenzgebiete. 3. Folge},
 series = {Ergeb. Math. Grenzgeb., 3. Folge},
 issn = {0071-1136},
 volume = {11},
 isbn = {3-540-22811-X},
 year = {2005},
 publisher = {Berlin: Springer},
 zbMATH = {2120947},
 Zbl = {1055.12003}
}

@article{GT09,
    author = {Ghioca, Dragos and Tucker, Thomas J.},
    title = {Periodic points, linearizing maps, and the dynamical {M}ordell-{L}ang problem},
    journal = {J. Number Theory},
    volume = {129},
    number = {6},
    pages = {1392--1403},
    year = {2009}
}

@article{GTZ07,
    author = {Guralnick, Robert M. and Tucker, Thomas J. and Zieve, Michael E.},
    title = {Exceptional covers and bijections on rational points},
    journal = {Int. Math. Res. Not. IMRN},
    number = {1},
    pages = {Art. ID rnm004, 20},
    year = {2007}
}

@article{GTZ08,
    author = {Ghioca, Dragos and Tucker, Thomas J. and Zieve, Michael E.},
    title = {Intersections of polynomial orbits, and a dynamical {M}ordell-{L}ang conjecture},
    journal = {Invent. Math.},
    volume = {171},
    number = {2},
    pages = {463--483},
    year = {2008}
}

@article{Jon08,
    author = {Jones, Rafe},
    title = {The density of prime divisors in the arithmetic dynamics of quadratic polynomials},
    journal = {J. Lond. Math. Soc. (2)},
    volume = {78},
    number = {2},
    pages = {523--544},
    year = {2008}
}

@article{MS94,
    author = {Morton, Patrick and Silverman, Joseph H.},
    title = {Rational periodic points of rational functions},
    journal = {Internat. Math. Res. Notices},
    number = {2},
    pages = {97--110},
    year = {1994}
}

@article{Stoll92,
 author = {Stoll, Michael},
 title = {Galois group over {{\(\mathbb{Q}\)}} of some iterated polynomials},
 fjournal = {Archiv der Mathematik},
 journal = {Arch. Math.},
 issn = {0003-889X},
 volume = {59},
 number = {3},
 pages = {239--244},
 year = {1992},
 doi = {10.1007/BF01197321},
 zbMATH = {119856},
 Zbl = {0758.11045}
}

@article{OdoniIterates,
    author = {Odoni, R. W. K.},
    title = {The {G}alois theory of iterates and composites of polynomials},
    journal = {Proc. London Math. Soc. (3)},
    volume = {51},
    number = {3},
    pages = {385--414},
    year = {1985}
}

@article{Pin04,
    author = {Pink, Richard},
    title = {On the order of the reduction of a point on an abelian variety},
    journal = {Math. Ann.},
    volume = {330},
    number = {2},
    pages = {275--291},
    year = {2004}
}

@article{Sil93,
    author = {Silverman, Joseph H.},
    title = {Integer points, {D}iophantine approximation, and iteration of rational maps},
    journal = {Duke Math. J.},
    volume = {71},
    number = {3},
    pages = {793--829},
    year = {1993}
}

@book{Sil09,
 author = {Silverman, Joseph H.},
 title = {The arithmetic of elliptic curves},
 edition = {2nd ed.},
 fseries = {Graduate Texts in Mathematics},
 series = {Grad. Texts Math.},
 issn = {0072-5285},
 volume = {106},
 isbn = {978-0-387-09493-9; 978-0-387-09494-6},
 year = {2009},
 publisher = {New York, NY: Springer},
 doi = {10.1007/978-0-387-09494-6},
 zbMATH = {5549721},
 Zbl = {1194.11005}
}

@article{SL96,
    author = {Stevenhagen, Peter and Lenstra, Hendrik W., Jr.},
    title = {{C}hebotarev and his density theorem},
    journal = {Math. Intelligencer},
    volume = {18},
    number = {2},
    pages = {26--37},
    year = {1996}
}

@article{Faltings,
 author = {Faltings, G.},
 title = {Finiteness theorems for abelian varieties over number fields.},
 fjournal = {Inventiones Mathematicae},
 journal = {Invent. Math.},
 issn = {0020-9910},
 volume = {73},
 pages = {349--366},
 year = {1983},
 doi = {10.1007/BF01388432},
 url = {https://eudml.org/doc/143051},
 zbMATH = {3944027},
 Zbl = {0588.14026}
}

@article{McQuillan,
 author = {McQuillan, Michael},
 title = {Division points on semi-abelian varieties},
 fjournal = {Inventiones Mathematicae},
 journal = {Invent. Math.},
 issn = {0020-9910},
 volume = {120},
 number = {1},
 pages = {143--159},
 year = {1995},
 doi = {10.1007/BF01241125},
 url = {https://eudml.org/doc/144272},
 zbMATH = {746072},
 Zbl = {0848.14022}
}

@article{Vojta,
 author = {Vojta, Paul},
 title = {Integral points of subvarieties of semiabelian varieties. {II}},
 fjournal = {American Journal of Mathematics},
 journal = {Am. J. Math.},
 issn = {0002-9327},
 volume = {121},
 number = {2},
 pages = {283--313},
 year = {1999},
 doi = {10.1353/ajm.1999.0014},
 url = {muse.jhu.edu/journals/american_journal_of_mathematics/toc/ajm121.2.html},
 zbMATH = {1275702},
 Zbl = {1018.11027}
}

@book{Xie17,
 author = {Xie, Junyi},
 title = {The dynamical {Mordell}-{Lang} conjecture for polynomial endomorphisms of the affine plane},
 fseries = {Ast{\'e}risque},
 series = {Ast{\'e}risque},
 issn = {0303-1179},
 volume = {394},
 isbn = {978-2-85629-869-5},
 year = {2017},
 publisher = {Paris: Soci{\'e}t{\'e} Math{\'e}matique de France (SMF)},
 language = {English},
 zbMATH = {6832236},
 Zbl = {1509.37136}
}

@book{DMLBook,
 author = {Bell, Jason P. and Ghioca, Dragos and Tucker, Thomas J.},
 title = {The dynamical {Mordell}-{Lang} conjecture},
 fseries = {Mathematical Surveys and Monographs},
 series = {Math. Surv. Monogr.},
 issn = {0076-5376},
 volume = {210},
 isbn = {978-1-4704-2408-4; 978-1-4704-2908-9},
 year = {2016},
 publisher = {Providence, RI: American Mathematical Society (AMS)},
 language = {English},
 doi = {10.1090/surv/210},
 zbMATH = {6583248},
 Zbl = {1362.11001}
}

@article{Amerik11,
 author = {Amerik, Ekaterina},
 title = {Existence of non-preperiodic algebraic points for a rational self-map of infinite order},
 fjournal = {Mathematical Research Letters},
 journal = {Math. Res. Lett.},
 issn = {1073-2780},
 volume = {18},
 number = {2},
 pages = {251--256},
 year = {2011},
 language = {English},
 doi = {10.4310/MRL.2011.v18.n2.a5},
 zbMATH = {6026553},
 Zbl = {1241.14011}
}

@article{etaleDML,
 author = {Bell, J. P. and Ghioca, D. and Tucker, T. J.},
 title = {The dynamical {Mordell}-{Lang} problem for {\'e}tale maps},
 fjournal = {American Journal of Mathematics},
 journal = {Am. J. Math.},
 issn = {0002-9327},
 volume = {132},
 number = {6},
 pages = {1655--1675},
 year = {2010},
 language = {English},
 doi = {10.1353/ajm.2010.0014},
 zbMATH = {5834745},
 Zbl = {1230.37112}
}

\end{document}